\documentclass [12pt,a4paper,reqno]{amsart}

\usepackage{etex}
\usepackage{xypic}
\xyoption{curve}
\input xy \xyoption{all}

\usepackage{amsmath,amsthm}
\usepackage{amssymb,amscd}
\input xy
\xyoption{all}

\usepackage{latexsym}
\usepackage{graphicx}
\usepackage{xypic}
\usepackage{etex}

\def\nSet{\{ 0,\dots, n\}}

\newcommand{\etype}[1]{\renewcommand{\labelenumi}{(#1{enumi})}}

\def\ealph{\etype{\alph} \dispace }
\def\dispace{\setlength{\itemsep}{2pt}}
\newcommand{\ds}[1]{\ {#1} \ }
\newcommand{\dss}[1]{\quad {#1} \quad }

\def\Iff{\ \Leftrightarrow \ }

\def\opr{\overline{\partial}}
\newcommand{\To}{\longrightarrow }

\def\Onto{\; -\hskip-5pt\twoheadrightarrow}

\def\mfB{\mathfrak{B}}
\def\mfD{\mathfrak{D}}
\def\mfL{\mathfrak{L}}
\def\mfM{\mathfrak{M}}

\def\tlR{\widetilde{R}}

\def\tlV{\widetilde{V}}
\def\tlW{\widetilde{W}}

\def\00{\{0\}}

\def\olR{\overline{R}}

\def\sm{\setminus}

\def\nucong{\cong_\nu}
\def\noi{\noindent}

\def\pSkip{\vskip 1.5mm \noindent}
\def\semiring0{semiring$^{\dagger}$}

\newcommand{\Ray}{\operatorname{Ray}}
\newcommand{\ray}{\operatorname{ray}}
\newcommand{\CS}{\operatorname{CS}}
\newcommand{\an}{{\operatorname{an}}}

\newcommand{\Loc}{\rm Loc}

\def\tT{\mathcal T}

\def\tG{\mathcal G}

\newtheorem{thm}{Theorem} [section]
\newtheorem*{thm*}{Theorem}
\newtheorem{cor}[thm]{Corollary}
\newtheorem{lem}[thm]{Lemma}

\newtheorem{prop}[thm]{Proposition}

\newtheorem*{claim*} {Claim}

\newtheorem*{theorem13.5'} {Theorem 13.5$'$}
\newtheorem{acknowledgment*}[thm] {Acknowledgment}
\newtheorem{example}[thm]{Example}
\newtheorem{examp}[thm]{Example}

    \newtheorem*{remarks*} {Remarks}
 \newtheorem{comm}[thm]{Comment}
 \newtheorem*{remark*}{Remark}
 \newtheorem{defn}[thm]{Definition}
\newtheorem{construction}[thm]{Construction}
\newtheorem{convention}[thm]{Convention}

\newtheorem{schol}[thm]{Scholium}

\newtheorem{problem}[thm]{Problem}

\newtheorem*{notation*} {Notation}

\newtheorem{rem}[thm]{Remark}

\def\N{\mathbb{N}}
\def\Z{\mathbb{Z}}

\def\RR{\mathbb{R}}

 \renewcommand{\sectionmark}[1]{}

\newcommand{\bfem}[1]{\textbf{#1}}

\newcommand{\veps}{\varepsilon}
\newcommand{\al}{\alpha}
\newcommand{\bt}{\beta}
\newcommand{\gm}{\gamma}

 \newcommand{\dl}{\delta}

\newcommand{\lm}{\lambda}

\newcommand{\zt}{\zeta}
\newcommand{\sig}{\sigma}
\newcommand{\vrp}{\varphi}

\def\simr{\sim_{\operatorname{r}}}

\begin{document}

\title[Stratifications of the ray space]{Stratifications of the ray space \\[1mm] of a tropical quadratic form\\[1mm] by Cauchy-Schwartz functions}

 \author[Z. Izhakian]{Zur Izhakian}
\author[M. Knebusch]{Manfred Knebusch}







\begin{abstract}

Classes of an equivalence relation on a module $V$ over a supertropical semiring, called rays, carry the underlaying structure of ``supertropical trigonometry'' and thereby a version of  convex geometry which is compatible with quasilinearity.  In this theory the traditional Cauchy-Schwarz inequality is replaced by the  CS-ratio which gives rise to special characteristic functions, called  CS-functions. These functions partite the ray space $\Ray(V)$ into convex sets and establish a main  tool for analyzing  varieties  of quasilinear stars in  $\Ray(V)$. They provide stratifications of $\Ray(V)$ and therefore a finer convex analysis that  helps for a better geometric understanding.

\end{abstract}

\maketitle
\setcounter{tocdepth}{1}
{ \small \tableofcontents}

\numberwithin{equation}{section}

\section*{Introduction}

Quadratic forms lay the foundation for various theories, taking a major role in their  studies. They play a similar role in supertropical mathematics \cite{QF1,QF2,Quasilinear,VR1,UB}, which is carried over the ``weaker'' structure of semirings.  As in classical theory, quadratic forms and  their bilinear companions lead to theories of  tropical trigonometry and convex geometry, where, due to the ``weak'' semiring structure,  the Cauchy-Schwarz inequality is  replaced by the Cauchy-Schwarz ratio, written CS-ratio. This type of convex geometry takes place over the space of rays -- equivalence classes of a suitable equivalence relation  \cite{Quasilinear}, and utilizes special characteristic functions, called CS-functions,  that emerge from the CS-ratio on ray spaces. The  CS-functions provide a useful tool for convex analysis, which is of much help in  understanding the variety of quasilinear stars in ray spaces  \cite{CSFunctions}.
This paper proceeds the study of CS-functions and the induced stratification of ray spaces.

 Classical quadratic forms can be explored via their supertropical images, which preserve some characteristic properties.
  These images are obtained by the use of suitable valuations and are easier to be classified  and investigated   in the supertropical setup where CS-functions appear. This approach becomes very effective when different bases are considered,   in particular when dealing with a system of quadratic forms \cite{CSFunctions}.


\subsection*{Supertropical semirings} $ $ \pSkip
Supertropical semirings carry  a rich algebraic structure   \cite{zur05TropicalAlgebra,nualg,IzhakianKnebuschRowen2010Linear,IR1,IR2} and provide the underlying structure of our framework.
A \textbf{supertropical semiring} is  a unital semiring~ $R$ with idempotent element $e:= e+e =1+1$
 such that,   for
all~ $a,b\in R$,  $a+b\in\{a,b\}$ whenever $ea \ne eb$ and  $a+b=ea$ otherwise. Consequently, $ea=0  \Rightarrow a =0$. The element~ $e$ determines the \textbf{ghost map} $\nu: a\mapsto ea$, and the \textbf{ghost ideal} $eR$ of~ $R$, which is a unital bipotent semiring,  i.e., $a+b\in \{ a,b\} $ for any $a,b\in eR$, and therefore totally ordered by
\begin{equation*}\label{eq:0.5}
a\le b \ds \Leftrightarrow a+b=b.
\end{equation*} Thereby, $R$ is equipped  with the $\nu$-ordering and  the $\nu$-equivalence:
\begin{equation*}\label{eq:nuorderring}
a <_\nu b \ds \Iff ea < eb, \qquad
a \nucong b \ds\Iff ea = eb,
\end{equation*}
 which  determine the addition of~$R$:
\begin{equation*}\label{eq:0.6}
a+b =\begin{cases} b&\ \text{if}\ a <_\nu b,\\
a&\ \text{if}\ a>_\nu  b,\\
eb&\ \text{if}\ a \nucong b.
\end{cases}
\end{equation*}
 The set
$\tT:=R\sm(eR)$ consists of the  \bfem{tangible} elements of $R$, while the set $\tG :=(eR)\sm\{0\}$ contains  the \bfem{ghost elements}. 
Nevertheless, the zero $0 = e0$ is regarded mainly as a ghost.
The semiring $R$ itself is said to be  \bfem{tangible}, if $e\tT=\tG$, i.e., $\tT$ is generates $R$  as
a semiring.
$R$  is a \textbf{supertropical semifield}, if in addition   both $\tT$ and $\tG$ are multiplicative abelian groups~ \cite[\S7]{IzhakianKnebuschRowen2010Linear}.

We stay in a purely tropical setting, but since later we intend to use results from the supertropical theory, we retain the supertropical notations, assuming that $R = eR$ is a bipotent semifield, $R = \{ 0 \}  \cup \tG$ with $\tG$ a totally ordered abelian group.
%
%
For formal reason we enlarge $R$ by an element $\infty$ to obtain a totally ordered set $\olR =R \cup \{ \infty\} $, with $x < \infty$ for all $x \in R$, on which the group $\tG$ acts by multiplication with orbits $\tG$, $\{ 0 \}$, $\{ \infty \}$. We further extend the automorphism $x \mapsto x^{-1}$ by putting
$0^{-1} = \infty$, $\infty^{-1} = 0$. The addition of~ $R$ extends to addition $\olR \times \olR \to \olR$ by the rule
$$x + y := \max \{ x,y \}. $$
We usually write $\olR = [0,\infty] $, where $R = [0,\infty [$ is a bipotent semified. The product $0 \cdot \infty$ is not defined.

At various places it will by helpful to have a bipotent semifield $\tlR \supset R$ in which for every $\lm \in R$ and $n \in \N$ there exists $\mu \in \tlR$ such that  $\mu^n = \lm.$

Recall that for every $n \in \N$ the ``$n$'th Frobenius map'' $\vrp_n: R \to R$, $\vrp_n(\lm) = \lm^n$, is an isomorphism form $R$ to a subsemifield $R^n$ of $R$. It follows that there is a bipotent semifield $R^{\frac 1n} \supset R$,  unique up to isomorphism over $R$,  such  that
\begin{equation*}\label{eq:1.7}
 R = \{ \lm^n \ds | \lm \in R^{\frac 1n}\}. \end{equation*}
(As already done for $n=2$ in \cite[\S3]{CSFunctions}.) We obtain a well defined bipotent semifield
\begin{equation*}\label{eq:1.8}
 \tlR = \bigcup_{n \in \N} R^{\frac 1 n} \supset R,\end{equation*}
where $R^{\frac 1 n} \subset R^{\frac 1 m}$  if $m | n$. We call $\tlR$ the \textbf{root closure} of $R$.
 For a given $\lm \in R$, $n \in \N$ we often denote the unique element $\mu$ of $\tlR$ with $\mu^n = \lm$ by $\sqrt[n]{\lm}$  or $\lm^{\frac 1n}$. For every $m \in Z$ we have the formula
$\sqrt[n]{\lm^m} = \sqrt[m]{\lm^n}.$

An $R$-module $V$ over a commutative supertropical semiring $R$  is defined in the familiar way.
A \textbf{ray} in $V$ is a class $\neq \00$ of the equivalence relation $x \simr y ,$ if  $\lm x = \mu y$ for some  $\lm, \mu \in R \sm \00$. We write $X =\ray(x)$ for the ray of a vector $x \in V \sm \00$, while  $\Ray(V)$ denotes the set of all rays in~ $V$, called the \textbf{ray space} of $V$.

\subsection*{CS-functions} $ $ \pSkip
In the sequel  we assume that a quadratic pair $(q,b)$ is given on
 $V$, i.e., a quadratic form $q: V \to R$, and a bilinear companion, i.e., $b:V \times V \to R$, satisfying
$$q(x+y) = q(x) + q(y) + b(x,y)$$
for all $x,y \in V$.
We further assume (up to \S\ref{sec:4}) that $q$ is anisotropic, i.e., $q(x) \neq 0$ for $x \neq 0$. For every pair $(x,y)$ in $V \sm \00$ we then have the \textbf{CS-ratio}
$$ \CS(x,y) := \frac{b(x,y)^2}{q(x)q(y)} \in R.$$
This CS-ration only depends on $\ray(x)$ and $\ray(y)$. Consequently, we define
$$ \CS(X,Y) := \CS(x,y)$$
for any rays $X,Y$ and vectors $x \in X$, $y \in Y$.

A \textbf{CS-function} on the ray space $\Ray(V)$ is a map $$f:\Ray(V) \To R,$$ for which there exists a ray $W\in \Ray(V)$ such that
$f(X) = \CS(X,W)$ for all $X \in \Ray(V)$.
CS-functions determine  functions $$ f_w: [0,\infty] \To R, \qquad w \in V \sm \00, $$ which  are of much help to understand parameterizations of intervals $[Y_1, Y_2]$ by elements of~ $\olR$, cf. \cite[\S7]{QF2} to be revised below in \S\ref{sec:1}. Each  function $f_w$ induces a subdivision of $[0,\infty]$, over which the behaviour of $f_w$ can be described explicitly by a monomial $\gm \lm^i$ (Proposition~ \ref{prop:1.5} and Theorem \ref{thm:1.7.a}). So these functions  are special cases of piecewise monomial functions, which  can be compared (Theorem \ref{thm:1.6}) and are carefully analyzed below (Theorem \ref{thm:1.7}).

Relaying this fine analysis and profiles of the CS-functions $\CS(W,-)$ on  a fixed closed interval ${[Y_1,Y_2]}$, cf. \cite[\S4 and \S5]{CSFunctions}, we introduce a much more general partition of $\Ray(V)$ into convex sets which supports the notions ``basic types'', ``relaxations'',  ``composed types'', and ''separation''.
This study leads to a  Sign Changing Theorem (Theorem \ref{thm:2.13}), which serves as a main tool in our advanced convex  analysis.

``Direct derivations'', as defined below (Definition \ref{def:2.7}),  provide a notion of neighbors of a stratum, and thereby a sequence of steps for a possible passage from a convex set to another one, considering  their types. Such steps are specified by junctions and butterflies, relaying on the notion of a regular ray, and give a systematic process to build a sequence of steps (Construction~ \ref{const:4.15}, derived from Theorem \ref{thm:4.13}, Corollary \ref{cor:4.14}, and Proposition ~\ref{prop:4.17}). It gives  instances of special sets assigned to each endpoint of an interval,  which can be enlarged by a widely used formal   saturation process, preserving basic inclusion properties (Theorem~ \ref{thm:4.19}).
The construction is extended in \S\ref{sec:4} to closed intervals with isotropic endpoints, which contain an anisotropic ray in their interior, and in consequence consist entirely of such rays.

\subsection*{An outlook} $ $ \pSkip
Much of the present paper discusses  functions $f: \Ray(V) \to R$ which are linear combinations of CS-functions with coefficients in $R \sm \00$. We denote the set of all these functions by $\mfM.$ Given a closed interval $[Y_1,Y_2]$, $Y_1 \neq Y_2$, we may ``compare'' the restrictions of two such functions $f$ and $g$ to $[Y_1, Y_2]$. It turns out that there is a succession of rays
$$ Z_0 = Y_1 <_{Y_1}Z_1  <_{Y_1} \cdots  <_{Y_1} Z_r = Y_2$$
such that on each open interval $]Z_{i-1}, Z_i[$ everywhere $f <g $, $f= g$, or $f >g$.
Here we employ the total ordering $ <_{Y_1}$ on $[Y_1,Y_2]$ appearing in \cite[\S7]{QF2}; that is
$$ Z \leq_{Y_1} Z' \dss\Leftrightarrow [Y_1,Z] \subset [Y_1, Z'].$$
(There is  a second total ordering $<_{Y_2}$ on $[Y_1, Y_2] = [Y_2, Y_2]$, reverse to $<_{Y_1}$.)

If, say, $f < g$ on $]Z_{i-1},Z_i[$,  we have $f(Z_{i-1})< g(Z_{i-1})$ or $f(Z_{i-1}) = g(Z_{i-1})$, and $f(Z_{i})< g(Z_{i})$ or $f(Z_{i}) = g(Z_{i})$. Thus, the set
$\{ X \in [Z_{i-1},Z_i] \ds | f(X) < g(X)\} $ is one of the sets $[Z_{i-1},Z_i]$, $]Z_{i-1},Z_i]$, $[Z_{i-1},Z_i[$, $]Z_{i-1},Z_i[$ . These four options are exclusive, since $R$ is a nontrivial dense totally ordered semifield, cf. e.g. \cite[Proposition 8.1]{QF2}.
If $f >g $ on $]Z_{i-1},Z_i[$, we face the same situation with $f$ and $g$ interchanged.

This panorama  leads to a seemingly large variety of membership problems in $\mfM$. Given  subsets $\mfB \subset \mfL$ of $\mfM$, we look for a set $\mfD \subset \mfM$, consisting of ``simple functions'' in some sense, and a set $J$ of closed intervals $[Y_1, Y_2]$ in $\Ray(V)$, such that a given $f \in \mfL$ can be tested  to be a member of $\mfB$ by comparing the restrictions  of $f$ and functions $g \in\mfD$ on these intervals $[Y_1, Y_2]$. Such a pair  $(\mfD,J)$ is termed a \textbf{toolbox} for $(\mfB,\mfL)$. The central question is: when does there exist an efficient toolbox  for $(\mfB,\mfL)$?

Most CS-functions $\CS(W,-)$ on $\Ray(V)$ can be combined linearly from ``simpler'' CS-function $\CS(W_i,-)$ in many ways. Indeed, choosing a vector $w \in W$ and taking a linear combination $w = \lm_1 w_1 + \cdots + \lm_r w_t$ with $w_i \in V \sm \00$,
$\lm_i \in R \sm \00$, let $W_i \in \ray(w_i)$ and
$$ \al_i := \frac{q(\lm_i w_i)}{q(w)} = \frac{\lm_i^2 q( w_i)}{q(w)} \leq e.$$
Then
$$ \CS(W,X) = \sum_{i=1}^r \al_i\CS(W_i,X)$$
for any $X \in \Ray(V)$, cf.  \cite[Lemma 5.1]{QF2}. This formula underscores  the importance to fix $\mfB$ and $\mfL$ in $ \mfM$ precisely for a meaningful membership problem.


\section{Uniqueness in the parametrization of a closed ray interval}\label{sec:1}

Let $Y_1$ and $Y_2$ be two different rays on $V$. Choosing vectors $\veps_1 \in Y_1$, $\veps_2 \in Y_2$, we have a surjective map
$$ \pi_{\veps_1, \veps_2}: [0,\infty] \Onto [Y_1, Y_2],$$
defined by
$$ \pi_{\veps_1, \veps_2}(\lm) := \ray(\veps_1 + \lm \veps_2), $$
cf. \cite[\S7]{QF2}. (Read $\pi_{\veps_1, \veps_2}(\infty) = Y_2$.)
Furthermore,  there is a partial ordering $\leq_{Y_1}$ on $\Ray(V)$ defined by
$$ Z \leq_{Y_1} Z' \dss \Iff [Y_1, Z] \subset [Y_1, Z'],$$
which restricts to a total ordering on $[Y_1, Y_2]$, cf. \cite[\S8]{QF2}. It is evident that $\pi_{\veps_1, \veps_2}$ is an increasing map from $[0, \infty]$ to $[Y_1,Y_2]$ with respect to the two total orderings. Thus, it is a priori  clear that for every $Z \in [Y_1, Y_2]$ the fiber $\pi^{-1}_{\veps_1, \veps_2}$ is a convex subset of $[0,\infty].$

We start out to determine some of these fibers explicitly by a careful look at the CS-functions restricted to $[Y_1, Y_2]$. Given a vector $w \in V \sm \00$, we have
\begin{equation}\label{eq:a.1.1}
  \CS(\veps_1 + \lm \veps_2, w) := \frac{b(\veps_1,w)^2 + \lm^2 b(\veps_2,w)^2 }{(\al_1 + \al_{12}\lm + \al_2 \lm ^2) q(w)},
\end{equation}
with $\al_1 := q(\veps_1)$, $\al_2 := q(\veps_2)$, $\al_{12} := b(\veps_1, \veps_2)$. In the following we assume for simplicity \emph{that $R$ is square closed}, i.e., $R = R^{\frac 12}.$ Note that now the total ordering on $\olR = [0,\infty]$  is dense, since $\lm < \mu $ implies $\lm < \sqrt{\lm \mu } < \mu.$

The function $\CS(\veps_1 + \lm \veps_2, w)$ is zero on the entire  set $[0,\infty]$, if $b(\veps_1,w) = b(\veps_2,w) = 0$, and then will be of no use for us. Therefore, we consider only  the vectors
\begin{equation}\label{eq:a.1.2}
  w \in V \sm \veps_1^\perp \cap \veps_2^\perp,
\end{equation}
where $\veps_i^\perp := \{ x \in V \ds | b(\veps_i, x) = 0\} $.
For these vectors $w$ we abbreviate
\begin{equation}\label{eq:a.1.3}
  f_w (\lm) := \frac{b(\veps_1,w)^2 + \lm^2 b(\veps_2,w)^2 }{(\al_1 + \al_{12}\lm + \al_2 \lm ^2) q(w)},
\end{equation}
 and obtain functions $f_w:[0,\infty] \to R$  which are nowhere zero on $]0,\infty[$ .
\begin{defn}\label{def:1.1}
We say that an $\olR$-valued function $f: C \to \olR$ on a convex subset  $C$ of $[0,\infty]$ is \textbf{monomial}, if
\begin{equation}\label{eq:a.1.4}
f(\lm) = \gm \lm^j
\end{equation}
for $\lm \in C$,  $j \in \Z$, and fixed $\gm \in R \sm \00$.
We call $j$ the \textbf{monomial degree} of $f$. Alternatively, we say that $f$ is $\lm^j$-monomial.
\end{defn}

Note that then $f $ avoids the values $0$ and $\infty $ on $C \cap \; ]0,\infty[$ . Furthermore, $f$ is either strictly increasing ($j > 0$) or strictly decreasing ($j < 0$), or constant ($j=0$).
\begin{defn}\label{def:1.2}
We call a function $f:[u,v] \to \olR$ on a subinterval $[u,v]$ of $[0,\infty]$ \textbf{piecewise monomial}, or  \textbf{pm} for short, if there is a \textbf{monomial subdivision} of $[u,v]$, i.e., a finite sequence
\begin{equation}\label{eq:a.1.5}
u = \al_0 < \al_1 < \cdots < \al_r = v
\end{equation}
in $R$, such that every restriction $f|[\al_{s-1},\al_s]$, $1 \leq s \leq r$, is monomial.
\end{defn}
We then have a sequence $(j_1, \dots, j_r)$ of the monomial degrees of $f$ on the intervals
$[\al_{s-1},\al_s]$, $1 \leq s \leq r$. If in this sequence $j_{s} = j_{s+1}$ for some $1 \leq s \leq r$, then $f$ is monomial of degree $j_{s} = j_{s+1}$ on $[\al_{s-1},\al_s]$, and so we can omit the point $\al_s$ in \eqref{eq:a.1.5}.  Repeating this process we finally obtain a subdivision \eqref{eq:a.1.5} of $[u,v]$ where $j_{s} \neq j_{s+1}$ for the monomial degrees of adjacent intervals
$[\al_{s-1},\al_s]$, $[\al_{s},\al_{s+1}]$. We call the obtained sequence \eqref{eq:a.1.5} the \textbf{reduced monomial subdivision} of $[u,v]$ with respect to the function  $f:[u,v ]\to  \olR $, and the sequence
$(j_1, \dots, j_r)$ of associated degrees the \textbf{reduced degree sequence} of $f$.

It easily seen that the reduced  monomial sequence and the reduced degree sequence are invariants of the pm function. (Hint: Given two subdivisions \eqref{eq:a.1.5}, first pass to a common refinement.)

\begin{rem}\label{rem:1.3}
  Given two pm functions $f:[u,v] \to \olR$ and $g:[u,v] \to \olR$, and a monomial subdivision \eqref{eq:a.1.5} of $[u,v]$ for $f$ and for $g$, we can refine both subdivisions to a common monomial subdivision.  Considering this common monomial subdivision, it is clear that the functions
  $\frac 1 f$, $f+g = \max(f,g)$, $f \wedge g = \min(f,g)$, $fg$ are again pm. \{These functions are defined point-wise  in the naive sense $(f+g) (\lm) = f(\lm) + g(\lm)$ etc.\} Obviously
  \begin{equation}\label{eq:a.1.6}
(f +g )\cdot(f \wedge g) = fg.
\end{equation}
\end{rem}

We return to the function $f_w$ from \eqref{eq:a.1.3}, always with $w \in V \sm \veps_1^\perp \cap \veps_2^\perp$,  which  are easily seen to be pm. We discuss this in details. The nominator $b(\veps_1,w)^2 + \lm^2 b(\veps_2,w)^2$ in \eqref{eq:a.1.3} is constant on
$[0, \frac{b(\veps_1,w)}{b(\veps_2,w)}]$  and $\lm^2$-monomial   on
$[\frac{b(\veps_1,w)}{b(\veps_2,w)}, \infty]$, provided that $\frac{b(\veps_1,w)}{b(\veps_2,w)} \neq 0$, which occurs if $b(\veps_1,w) = 0$ and $b(\veps_2,w)> 0$. In this case
$b(\veps_1,w)^2 + \lm^2 b(\veps_2,w)^2$ is $\lm^2$-monomial on $[0,\infty]$.

Considering the function
  \begin{equation}\label{eq:a.1.7}
  q(\veps_1 + \lm \veps_2) = \al_1 + \al_{12}\lm + \al_2 \lm ^2
\end{equation}
on $[0,\infty]$, we distinguish two cases:
\begin{description}\dispace
  \item[Case 1] $\al_1 \al_2 < \al_{12}^2$,
  \item[Case 2] $\al_1 \al_2 \geq \al_{12}^2$.
\end{description}

In Case 1 the term $\al_1$ in \eqref{eq:a.1.7} is dominant iff $\al_{12} \lm \leq \al_1$
and $\al_{2} \lm^2  \leq \al_1$, i.e., $\lm \leq \frac{\al_1}{\al_{12}} \wedge \sqrt{\frac{\al_1}{\al_{2}}}$. This simplifies to $\lm \leq \frac{\al_1}{\al_{12}}$, since $\frac{\al_1^2}{\al_{12}^2} < \frac{\al_1}{\al_{2}}$.
The term $\al_{12} \lm $ in \eqref{eq:a.1.7} is dominant iff $\al_1 \leq \al_{12} \lm $
and $\al_{2} \lm^2  \leq \al_1 \lm $, i.e., $ \frac{\al_1}{\al_{12}} \leq \lm \leq   {\frac{\al_1}{\al_{2}}}$. Finally,
the term $\al_2 \lm^2$  is dominant iff $\al_{1} \leq \al_2 \lm ^2$
and $\al_{12} \lm  \leq \al_2 \lm ^2$, i.e., $ \sqrt{\frac{\al_1}{\al_{2}}} + \frac{\al_{12}}{\al_{2}}  \leq \lm $, which simplifies to $ \frac{\al_{12}}{\al_{2}} \leq \lm$.

In Case 2 the quadratic form $q|R \veps_1 + R \veps_2$ is quasilinear, and so
$$ \al_1 + \al_{12}\lm + \al_2 \lm ^2  = \al_1 +  \al_2 \lm ^2 .$$
This function is constant on $[0, \sqrt{\frac{\al_1}{\al_2}}]$
and monomial of degree 2 on $[\sqrt{\frac{\al_1}{\al_2}}, \infty]$.

We summarize this, also for later use.
\begin{lem}\label{lem:1.4} $ $
  \begin{enumerate}\ealph
    \item If $\al_1 \al_2 < \al_{12}^2$, then  the function
    $q(\veps_1 + \lm \veps_2)$ is constant on
    $[0, {\frac{\al_1}{\al_{12}}}]$, monomial of degree 1 on
    $[\frac{\al_1}{\al_{12}}, {\frac{\al_{12}}{\al_2}}]$, and monomial of degree 2 on
    $[\frac{\al_{12}}{\al_2}, \infty]$.

     \item If $\al_1 \al_2 \geq  \al_{12}^2$, then  the function
    $q(\veps_1 + \lm \veps_2)$ is constant on
    $[0, \sqrt{\frac{\al_1}{\al_{2}}}]$, and monomial of degree 2 on
    $[\sqrt{\frac{\al_{1}}{\al_2}}, \infty]$.

  \end{enumerate}
\end{lem}

The symmetry in these statements  is also clear from the fact that
\begin{equation}\label{eq:a.1.8}
\ray(\veps_1 + \lm \veps_2) = \ray(\veps_2 + \lm^{-1} \veps_1).
  \end{equation}

From the above analysis of the nominator and denominator in  \eqref{eq:a.1.8} we obtain the following.

\begin{prop}\label{prop:1.5} $ $
  \begin{enumerate}\ealph
    \item Let $\al_1 \al_2 < \al_{12}^2$.  Then  $f_w$ is constant on
    $$ A_w := \bigg[ 0, \ \frac{b(\veps_1,w)}{b(\veps_2,w)} \wedge \frac{\al_1}{\al_{12}} \bigg]$$
and on
    $$ C_w := \bigg[\frac{b(\veps_2,w)}{b(\veps_1,w)} + \frac{\al_{12}}{\al_2},  \ \infty \bigg].$$
    On
        $$ B_w := \bigg[ \frac{b(\veps_1,w)}{b(\veps_2,w)} \wedge \frac{\al_1}{\al_{12}}, \
        \frac{b(\veps_2,w)}{b(\veps_1,w)} + \frac{\al_{12}}{\al_{2}}  \bigg]$$
        $f_w$ is pm and not constant on any subinterval.

    \item Let $\al_1 \al_2 \geq  \al_{12}^2$.  Then  $f_w$ is constant on
    $$ A_w := \bigg[ 0, \  \frac{b(\veps_1,w)}{b(\veps_2,w)} \wedge\sqrt{ \frac{\al_1}{\al_{2}}} \bigg]$$
and on
    $$ C_w := \bigg[\frac{b(\veps_2,w)}{b(\veps_1,w)} + \sqrt{\frac{\al_{2}}{\al_1}}, \ \infty \bigg].$$
    On
        $$ B_w := \bigg[ \frac{b(\veps_1,w)}{b(\veps_2,w)} \wedge\sqrt{ \frac{\al_1}{\al_{2}}}, \
        \frac{b(\veps_2,w)}{b(\veps_1,w)} + \sqrt{\frac{\al_{2}}{\al_{1}}}  \bigg]$$
        $f_w$ is pm without being constant on any subinterval, provided that $B_w$ is not a singleton.
If $B_w$ is a singleton, then $f_w$ is constant on $[0,\infty]$.

  \end{enumerate}

\end{prop}

   \begin{proof}
     The nominator in \eqref{eq:a.1.3} has the reduced degree sequence $(0,2)$, while the dominator has the reduced degree sequence $(0,1,2)$ in Case 1, and $(0,2)$ in Case 2. Constance of~ $f_w$ happens on the convex sets, where the monomial degrees of the nominator and the denominator coincide. This gives all  claims.
   \end{proof}

 \begin{defn}\label{def:1.6} Given rays $Y_1, Y_2$, we say that the parameter $\lm_0 \in [0,\infty]$ is \textbf{unique to the right} for $(Y_1,Y_2)$, if
\begin{equation}\label{eq:a.1.9}
\ray(\veps_1 + \lm \veps_2) \neq
\ray(\veps_1 + \lm_0\veps_2)
  \end{equation}
  for $\lm > \lm_0$, $\veps_1 \in Y_1$, $\veps_2 \in Y_2$,  and that $\lm_0$ is \textbf{unique to the left} for $(Y_1, Y_2)$, if \eqref{eq:a.1.9} holds for $\lm < \lm_0$. If both properties are valid, i.e., $\pi^{-1}(\pi(\lm_0)) = \{ \lm_0\}$ for
  $\pi = \pi_{\veps_1,\veps_2}$, we say that $\lm_0$ is \textbf{unique} for $(Y_1, Y_2).$
\end{defn}
With this terminology Proposition \ref{prop:1.5} has the following obvious consequence.

\begin{thm}\label{thm:1.7.a}
  Let $Y_1$ and $Y_2$ be different rays in $V$, and in the notations of Proposition \ref{prop:1.5} let
  $B_w = [u_w, v_w]$ for $w \in V \sm \veps_1^\perp \cap \veps_2^\perp$ (e.g., if $\al_1 \al_2 < \al_{12}$, then $u_w = \frac{b(\veps_1,w)}{b(\veps_2,w)} \wedge \frac{\al_1}{\al_{12}}$).
  Then $\lm_0 \in [0, \infty]$ is unique to the right for $(Y_1,Y_2)$, if there exists some  $w \in V \sm \veps_1^\perp \cap \veps_2^\perp$ with $\lm_0 \in \; ]u_w, v_w] $, and $\lm_0$ is unique to the left, if there exists such vector $w$ with $\lm_0 \in [u_w, v_w[ $.
\end{thm}

%
%
%
%
\section{Piecewise monomial functions in the ray space}\label{sec:b}
In the following, for simplicity, we always assume   that $R$ is root closed. Let $C$ be a convex subset of $\Ray(V)$ and assume that $Y_1$ and $Y_2$ are different rays in $C$. After a choice of vectors $\veps_1 \in Y_1$ and $\veps_2 \in Y_2$, we have the parametrization
$$ \pi_{\veps_1, \veps_2}: [0,\infty]\Onto  [Y_1, Y_2],
\qquad \lm \mapsto \ray(\veps_1 + \lm \veps_2),$$
as studied in \S\ref{sec:1}.

\begin{defn}\label{def:b.1} We say that a function $F: C \to \olR$ is \textbf{monomial of degree $i$ on $[Y_1, Y_2]$}, or \textbf{$\lm^i$-monomial}, if the map
$$ F_{\veps_1, \veps_2}:= F \circ \pi_{\veps_1, \veps_2}:  [0,\infty]\Onto \olR$$
is $\lm^i$-monomial. We call $F$ \textbf{piecewise monomial} on $[Y_1,Y_2]$, or \textbf{pm} for short,  if $F_{\veps_1, \veps_2}$ is pm on $[0,\infty]$, as defined in \S\ref{sec:1}.
\end{defn}

\begin{examp}\label{exmp:b.2} For a ray $W$ in $V$ the $R$-valued function $X \mapsto \CS(X,W)$ on $\Ray(V)$ is pm on every interval $[Y_1, Y_2]$ with $Y_1 \neq Y_2$, for which $w \notin V \sm \veps^\perp_1 \cap \veps^\perp_2$ for fixed vectors $w \in W$, $\veps_2 \in Y_1$,  $\veps_2 \in Y_2$, as is clear from \S\ref{sec:1}, cf. Proposition \ref{prop:1.5}.

\end{examp}

\begin{comm}\label{com:b.3}
The functions $F_{\veps_1, \veps_2}$ depend on the choice of the vectors $\veps_1, \veps_2$ in $Y_1, Y_2$. If we replace  $\veps_1, \veps_2$ by other vectors $\veps_1', \veps_2'$  in $Y_1, Y_2$, then $\veps_1' = \rho_1 \veps_1$, $\veps_2' = \rho_2 \veps_2$ with $\rho_1, \rho_2 \in \tG$, and so
$$ \pi_{\rho_1 \veps_1, \rho_2 \veps_2}  (\lm ) = \pi_{ \veps_1, \veps_2}\bigg(\frac{\rho_1}{\rho_2} \lm\bigg), \qquad F_{\rho_1 \veps_1, \rho_2 \veps_2}  (\lm ) = F_{ \veps_1, \veps_2}\bigg(\frac{\rho_1}{\rho_2} \lm\bigg).$$
But the properties considered in Definition \ref{def:1.1} are independent of the choice of the $\veps_i$.  To avoid distraction by the dependence on the choice of the $\veps_i$, which for most issues is irrelevant, we change  our view point. We regard $Y_1$ and $Y_2$ as \textbf{pointed rays} with base points $\veps_i$, $Y_i = \tG \veps_1$, and write
$$ \pi_{Y_1,Y_2}: [0,\infty] \Onto [Y_1, Y_2], \qquad  F_{Y_1,Y_2} =  F \circ \pi_{Y_1,Y_2},$$
instead of $ \pi_{\veps_1,\veps_2}$. Thus we work with pointed rays, but abusively still call them ``rays''.
\end{comm}
If $F$ is piecewise monomial on $[Y_1,Y_2]$, then we can choose a finite sequence
\begin{equation}\label{eq:b.1}
  0 \leq \al_0 <  \al_1 < \cdots < \al_r \leq \infty
\end{equation}
in $[0, \infty]$ such that
$F(\al_0) = Y_1$, $F(\al_r) = \infty$, and $F|[\al_{s-1}, \al_s]$ is monomial of some degree $j_s$ for $1 \leq s \leq r$.
(Then $F$ takes constant value $Y_1$ on $[0,\al_0]$ and constant value $Y_2$ on $[\al_r,\infty]$.
We could always take $\al_0 = 0$, $\al_r = \infty$. But the present setting has more flexibility.)

We call \eqref{eq:b.1}  again a \textbf{monomial subdivision} of $[\al_0, \al_r]$ for $F$.
Any two such monomial subdivisions can be refined to a third  monomial subdivision.
On the other hand we can reduce \eqref{eq:b.1} to a sequence where adjacent monomial degrees are different, called the \textbf{reduced monomial subdivision} and so have a unique \textbf{reduced monomial degree sequence} for~ $F$ on $[Y_1, Y_2]$.
If \eqref{eq:b.1} is the  reduced monomial subdivision, then $[0,\al_0]$ is the unique maximal interval which $F$ has constant value $Y_1$.
As consequence of Remark \ref{rem:1.3} we obtain
\begin{rem}\label{rem:b.4}
If
$F$ and $G$ are $R$-valued functions on a convex subset $C$ of $\Ray(V)$, and  are pm on a ray-interval $[Y_1, Y_2]$, then the functions $\frac 1 F$, $F+G = \max(F,G)$, $F \wedge G = \min(F,G)$, and $FG$ are again pm on $[Y_1,Y_2],$ and
$$ (F+G)(F \wedge G )= FG .$$
\end{rem}

 \begin{prop}\label{prop:b.5} If
  $F:C \to \olR$ is piecewise monomial or $\lm^i$-monomial
   on $[Y_1,Y_2]$, then $F$ has the same property on every subinterval $[Z_1,Z_2]$  of $[Y_1,Y_2]$.
 \end{prop}

\begin{proof} We start with a monomial subdivision \eqref{eq:b.1} of $[Y_1, Y_2]$ for $F$.
  We write

  $$ Z_1 = \ray(\veps_1 + \zt \veps_2), \qquad Z_2 = \ray(\veps_1 + \eta \veps_2), $$
  with $\al_0 \leq \zt < \eta \leq \al_r$, $\veps_1 \in Y_1$, $\veps_2 \in Y_2$.
Given $\mu \in \olR$, we have
  $$\begin{array}{ll}
      \ray\big((\veps_1 + \zt \veps_2)+\mu(\veps_1 + \eta \veps_2)\big) &=
\ray\big((1 + \mu)\veps_1 + (\zt + \mu \eta)\veps_2\big) \\[1mm] &=
\ray\big(\veps_1 + \frac{(\zt + \mu \eta)}{(1 + \mu)}\veps_2\big).
    \end{array}  $$
 Now observe that
 $$ \frac{(\zt + \mu \eta)}{(1 + \mu)} = \left\{
 \begin{array}{ll}
   \zt & \text{if } \mu \leq \frac \zt \eta,  \\[1mm]
   \mu \eta &  \text{if } \frac \zt \eta \leq \mu \leq  1,  \\[1mm]
   \eta &    \text{if }   1  \leq \mu.   \\
 \end{array}
 \right. $$
Thus
 $$ g(\mu) : = F_{Z_1, Z_2} (\mu) = \left\{
 \begin{array}{ll}
   f(\zt) & \text{if } 0 \leq \mu \leq \frac \zt \eta,  \\[1mm]
  f( \mu \eta) &  \text{if } \frac \zt \eta \leq \mu \leq  1,  \\[1mm]
  f( \eta) &    \text{if }   1  \leq \mu \leq \infty.   \\
 \end{array}
 \right. $$
Consequently, using the base points $\veps_1 + \zt \veps_2$, $\veps_1 + \eta \veps_2$ for $Z_1$, $Z_2$, we obtain a monomial subdivision for $F$ on $[Z_1, Z_2]$, starting with $0 < \frac{\zt}{\eta}$ and ending with $1< \infty $ for the parameter ~$\lm$. If $f(\lm) = \gm \lm^i$ for $f = F \circ \pi_{Y_1, Y_2}$ on a subinterval $[u,v]$ of $[\zt, \eta] \cap [\al_{0}, \al_r]$, then
\begin{equation}\label{eq:b.2}
  g(\mu) = \gm \eta^i \mu^i
\end{equation}
on $[u,v]$.
\end{proof}

\begin{thm}\label{thm:1.6}
Assume that $F: [Y_1,Y_2] \to \olR $ is a piecewise monomial function.
\begin{enumerate}\ealph
  \item The image of $F$ is a  closed subinterval $[\rho,\sig]$ of $[0,\infty]$. If $f = F_{Y_1,Y_2}$ and
$$ \al_0  < \al_1 < \al_2 < \cdots < \al_r $$ is a finite sequence such that  $f(\al_0)=Y_1$ and $f(\al_r)=Y_2$ in $[0,\infty]$, and $f$ is monomial on each interval $[\al_{s-1}, \al_s]$, $1 \leq s \leq r$, then
\begin{equation}\label{eq:1.10}
\rho = \min_{0 \leq s \leq r} f(\al_s), \qquad \sig = \max_{0 \leq s \leq r} f(\al_s).
\end{equation}

  \item If $\vrp$ is a piecewise monomial $\olR$-valued function on an interval $[u,v]$ in $[0,\infty]$, which contains $[\rho, \sig]$, then the function $\vrp \circ F: [Y_1,Y_2] \to \olR$ is again piecewise monomial.
\end{enumerate}
\end{thm}
\begin{proof} (a): The function $f =F_{Y_1,Y_2}$ maps each interval $[\al_{s-1}, \al_s]$ onto a closed interval $I_s = f([\al_{s-1}, \al_s])$, perhaps degenerated to a one--point set, and $I_{s-1} \cap I_s \neq \emptyset$ for $s > 0 $. Thus
$$f\big([\al_{s-1}, \al_s]\big) = \bigcup_{s=1}^n I_s$$
  is again a closed interval. More precisely, $I_s = [f(\al_{s-1}), f(\al_s)]$ if $f(\al_{s-1}) < f(\al_s)$,
  $I_s = [f(\al_{s}), f(\al_{s-1})]$ if $f(\al_{s}) < f(\al_{s-1})$, and $I_s = \{f(\al_s)\}$ if $f(\al_{s-1}) = f(\al_s)$. This implies  claim~ \eqref{eq:1.10}.

\pSkip
  \noindent (b): There is nothing to do if $\rho = \sig$. Assume that $\rho < \sig$. Replacing the pm function~ $\vrp$ by its restriction to $[\rho,\sig]$, we also assume that $[u,v] = [\rho, \sig].$ We now refine the set $\{ f(\al_0), \dots, f(\al_r)\} $ arranged to an ascending sequence
  $$ \rho = u_0 < u_1 < \dots < u_m = \sig$$
  by adding finitely many points in $[\rho,\sig]$ such that $\vrp$ is monomial on each interval $[u_{j-1}, u_j]$, $1 \leq j \leq m$. We prove that $\vrp \circ f$ is piecewise monomial on each interval $[\al_{s-1}, \al_s]$ and then will be done.

  If $f(\al_{s-1}) = f( \al_s)$, then $\vrp \circ f$ is constant on $[\al_{s-1}, \al_s]$. Assume that
  $f(\al_{s-1}) < f( \al_s)$. If there is no point $u_j$ in $]f(\al_{s-1}), f(\al_s)[$\;, then $\vrp \circ f $ is monomial on $[\al_{s-1}, \al_s]$. Otherwise we have a finite subsequence
  $$ f(\al_{s-1}) = u_k < u_{k+1} <  \cdots <  u_\ell = f(\al_s) $$ of $(u_0, \dots, u_m)$. The function $f$ maps $[\al_{s-1}, \al_s]$ bijectively onto $[f(\al_{s-1}), f(\al_s)]$, since $f$ is $\lm^i$-monomial on $[\al_{s-1}, \al_s]$ with $i > 0$ and $R^{\frac 1i} = R$. Thus there is a unique sequence
  $$ \al_{s-1} = \bt_{s,k} < \bt_{s,k+1} < \cdots < \bt_{s,\ell} =  \al_s$$ with $f(\bt_{s,t}) = \mu_t$ for $k \leq t \leq \ell$.  Now $\vrp \circ f$ is a monomial on each interval $[\bt_{s,t-1}, \bt_{s,t}]$, $k+1 \leq t \leq \ell$, since $f$ is monomial on this interval and $\vrp$ is monomial on its image $[u_{t-1}, u_t]$ under $f$. Thus $\vrp \circ f$ is pm on $[\al_{s-1}, \al_s]$.  If $f(\al_{s}) < f (\al_{s-1})$, we argue with
  $$ f(\al_{s}) = u_k < u_{k+1} <  \cdots <  u_\ell = f(\al_{s-1}) $$ and obtain again that $\vrp \circ f$ is pm on $[\al_{s-1}, \al_s]$.
\end{proof}

Our next goal is to ``compare'' $R$-valued pm functions on parts of a closed interval.

\begin{defn}\label{def:1.6} We say that two $R$-valued  functions $F$ and $G$ on a subset $D$ of $\Ray(V)$ are \textbf{comparable}, if $F \leq G$ or $F \geq G$ (everywhere) on $D$, and that $F,G$ are \textbf{strictly comparable}, if $F < G$, $F =G$, or $F >G$ (everywhere) on $D$. More generally we say that $R$-valued functions $F_1, \dots, F_r$ on $D$ are comparable (resp. strictly comparable), if any two of the $F_i$'s are comparable (resp. strictly comparable).
\end{defn}

\begin{thm}\label{thm:1.7}
Assume that $Y_1 \neq Y_2$ are rays in $V$,  that $F:[Y_1,Y_2] \to R $ is  a $\lm^i$-monomial function,  and that $G:[Y_1,Y_2] \to R $ is a  $\lm^j$-monomial function ($i,j \in \Z$).
  \begin{enumerate} \ealph
    \item If $i=j$, then $F$ and $G$ are strictly comparable.
    \item If $F$ and $G$ are comparable, but not strictly comparable, and $F(Y_1) \neq  G(Y_1)$,
    $F(Y_2) \neq  G(Y_2)$, then there is a   unique ray $Z$ in $]Y_1, Y_2]$ such that $F < G$ everywhere on $[Y_1, Z[$ and $F = G$ on $[Z, Y_2]$.
    \item  If $F$ and $G$ are not comparable, there is  a  unique ray $Z$ in $]Y_1, Y_2[$ such that $F$ and $ G$  are strictly comparable on  $[Y_1, Z[$ and on $]Z, Y_2]$.
  \end{enumerate}
\end{thm}

\begin{proof}
  Let  $f =F_{Y_1,Y_2}$, $g =G_{Y_1,Y_2}$. We work with monomial subdivisions for $F$ and $G$ on the whole set $[0, \infty]$. For $\lm \in [0,\infty]$ we have
  $f(\lm) = \gm \lm^i$, $g(\lm) = \dl \lm^j$ with fixed $\gm, \dl \in \tG$.

 If $i =j$, then trivially $F < G$ if $\gm < \dl$, $F = G$ if $\gm = \dl$, and  $F >G $ if $\gm > \dl$.
   Assume now that $i \neq j$.
For each sign $\Box \in \{  <, = , >\}$ and $\lm\in \; ]0,\infty[$ we have
\begin{equation}\label{eq:1.11}
  f(\lm) \ds\Box g(\lm) \dss \Leftrightarrow  \gm \lm^i\ds \Box \dl  \lm^j
  \dss \Leftrightarrow  \lm^{i-j}\ds \Box \frac \dl \gm .
\end{equation}
If $F$ and $G$ are not  comparable, then all three signs $< , = , > $  occur here, and we infer from ~ \eqref{eq:1.11} that $F$ and $G$ are strictly comparable on $[Y_1,Z [ $ and $]Z, Y_2]$ with
\begin{equation}\label{eq:q.1.12}
Z = \ray\bigg( \veps_1 + \sqrt[k]{\frac \dl \gm}  \veps_2\bigg)
=\ray\big( \sqrt[k]{\gm }\veps_1 +\sqrt[k]{\dl }  \veps_2\big)
\end{equation}
in case $i > j $
\begin{equation}\label{eq:1.13}
Z = \ray\bigg( \veps_1 + \sqrt[k]{\frac \gm \dl}  \veps_2\bigg)
=\ray\big( \sqrt[k]{\dl }\veps_1 +\sqrt[k]{\gm }  \veps_2\big)
\end{equation}
in case $i < j $, where $k = |i-j|$ in both cases.

If $F \leq G$, but not $F < G$, then on \eqref{eq:1.11} the signs $<$ and $=$ occur, but not $>$, and we infer that $F < G$ on $[Y_1, Z[$, $F = G$ on $[ Z, Y_2 [$ where $Z$ is again the ray in \eqref{eq:1.13}.
\end{proof}

\begin{rem}\label{rem:1.8}
  If $F$ and $G$ are comparable and $F<G$ everywhere on $[Y_1,Z[$, $F =G $
on $[Z,Y_2]$, as described in Theorem \ref{thm:1.7}.(b), then $[Y_1,Z[ $ is \textbf{not} a closed interval. In other words $[Y_1,Z'] \neq [Y_1,Z[$ for every ray $Z'$ in $[Y_1,Z[$. Likewise, for the ray $Z$ appearing in Theorem~ \ref{thm:1.7}.(c) neither $[Y_1,Z[$ nor $]Z,Y_2]$ is a  closed interval. All this is clear from ~ \eqref{eq:1.11} in the proof of Theorem~ \ref{thm:1.7}, since the function $\lm \mapsto \lm^{i-j}$ on $]0,\infty[$ is strictly monomial.

\end{rem}
If $F$ and $G$ are $R$-valued functions on a set $D \subset \Ray(V)$, we define the functions $\max(F,G)$ and
$\min(F,G)$ on $D$ in the obvious way,
$$\max:x \mapsto \max(F(x), G(x)), \qquad \min: x \mapsto \min(F(x), G(x)).$$ Note that
$\max(F,G) = F +G$ and that
\begin{equation}\label{eq:1.14}
  (F+G) \cdot \min(F,G) = FG.
\end{equation}

\begin{thm}\label{thm:1.9}
  Assume that  $F$ and $G$ are $R$-valued pm functions on $[Y_1,Y_2]$
   and that $R$ is root closed. Then the functions $F + G = \max(F,G)$ and $\min(F,G)$ are again pm on $[Y_1,Y_2]$.

\end{thm}

\begin{proof}
   Let  $f =F_{Y_1,Y_2}$, $g =G_{Y_1,Y_2}$.  We have a sequence $  \al_0 < \al_1 < \cdots < \al_r $ in $[0,\infty]$ with $f(\al_0) = g(\al_0) =Y_1$, $f(\al_r) = g(\al_r) =Y_2$, such that both $f$ and $g$ are monomial on each interval $[\al_{s-1}, \al_s]$.
   Let $Z_s := \ray(\veps_1 + \al_s \veps_2)$. By Theorem \ref{thm:1.7} we have rays $T_s$ in $[Z_{s-1},Z_s]$ such that $F$ and $G$ are comparable on $[Z_{s-1},T_s]$ and on $[T_s, Z_s]$, $1 \leq s \leq r$, and so  either $F +G = F $, $\min(F,G) = G$, or $F+G = G$, $\min(F,G) = F$ on each of these intervals.
\end{proof}

\begin{cor}\label{cor:1.10}
  Assume that $R$ is root closed, and that $F_1,\dots, F_r$ are piecewise monomial functions on a closed interval $[Y_1, Y_2]$. Then there is a finite sequence
  $$ Y_1 = Z_0 <_{Y_1} Z_1 <_{Y_1} \cdots <_{Y_1} Z_r = Y_2$$ in $[Y_1, Y_2]$ such that $F_1,\dots, F_r$ are strictly comparable on each open interval $]Z_{i-1}, Z_i[ $, $ 1\leq i \leq r$.
\end{cor}
\begin{proof} The case of $n=2$ is obtained from
  Theorem \ref{thm:1.7} by arguing as in the proof of Theorem ~\ref{thm:1.9}. The case of $r > 2$ follows by an easy induction.
\end{proof}

\section{Partitions of $\Ray(V)$ into convex sets by linear combinations of CS-functions}\label{sec:2}

In \cite[\S4 and \S5]{CSFunctions} we introduced a partition of $\Ray(V)$ into convex sets according to the profiles of the CS-functions $\CS(W,-)$ on  a fixed closed interval $\overrightarrow{[Y_1,Y_2]}$. We elaborate   this study by constructing much more general partitions of $\Ray(V)$ into convex sets, where the notion of ``basic types'', ``relaxations'', and ``composed types'' still make sense.  As before we assume that $R$ is root closed, $R = \tlR$, and that the quadratic from $q$ on the $R$-module $V$ is anisotropic.

\begin{defn}\label{def:2.1} Given a finite set of rays
$S = \{Y_1, \dots, Y_n \}$ in $V$, we call an $R$-valued function $f$ on $\Ray(V)$ \textbf{$S$-basic}, if $f$ is a linear combination of the CS-functions $\CS(Y_1,-), \dots,$ $ \CS(Y_n,-) $,
\begin{equation}\label{eq:2.1}
  f = \sum_{j= 1}^{n} \gm_j \CS(Y_j,-)
\end{equation}
with $\gm_j \in R = eR$.
 \end{defn}
 We pick a finite set $\mfB = \{ f_1, \dots, f_m\}$ of $S$-basic functions and study the partition
 of $\Ray(V)$ into nonempty sets
\begin{equation}\label{eq:2.2}
\bigcap_{1\leq i < j \leq m} \{ X \in \Ray(V) \ds | f_i(X) \ds {\Box_{i,j}} f_j(X) \}
\end{equation}
with signs $\Box_{i,j}$ in $\{ <, =, > \} $, called the \textbf{$\mfB$-partition} of $\Ray(V)$.

The presciption of the set $S$ is just a convention to fix ideas. Any finite set $\mfB$ in the $R$-module
$\sum_Y R \cdot \CS(Y,-)$ with $Y$ running through $\Ray(V)$ appears in this way.

\begin{defn}\label{def:2.3}
A \textbf{basic type} $T$ for a family $\mfB$, or $\mfB$-type  for short, is a conjunction
\begin{equation}\label{eq:2.3}
T = \bigwedge_{1\leq k < \ell \leq m}  f_k(X) \ds {\Box_{k,\ell}} f_\ell(X)
\end{equation}
where $X$ is a formal variable for rays in $V$, with signs $\Box_{k,\ell}$ in $\{ <, =, > \} $, which is not contradictory, i.e., can be satisfied by some ray in $V$.

A \textbf{relaxation}  $U$ of $T$ is a formula where some of the signs $<, >$ appearing in \eqref{eq:2.3} are replaced by $\leq, \geq$ respectively.  Then
\begin{equation}\label{eq:2.4}
Y = T_1 \vee \cdots \vee T_r
\end{equation}
with unique basic types $T_1, \dots, T_r$, up to permutations, called the \textbf{components} of $U$, $T$ being one of them.
Finally, we call a basic type $T'$, which is a component of some relaxation of $T$, a \textbf{type derived from} $T$, or a \textbf{derivate of} $T$.
\end{defn}

In parallel,  for each basic $\mfB$-type we introduce the  set
\begin{equation}\label{eq:2.5}
\{ T \} := \{ X \ds | X \models T\},
\end{equation}
by which we mean the subset of all rays in $V$ satisfying the formula $T$,
and  for a relaxation $U$ of $\{ T \}$ the set
\begin{equation}\label{eq:2.6}
\{ U \} := \{ X \ds | X \models U \},
\end{equation}
called respectively  the \textbf{stratum} of $T$ and the \textbf{relaxation set} of $U$.  (The word ``stratum'' will be justified later.)

Note that, if $U = T_1 \vee \cdots \vee T_r$
with basic types $T_1, \dots, T_r$, then
\begin{equation}\label{eq:2.7}
  \{ U \}  =\{ T_1 \} \cup \cdots \cup \{ T_r \}.
\end{equation}
The strata $\{ T_i \}$, with $T_i$ a component of some relaxation $U$ of $T$, are said to be \textbf{derived} from $\{ T \}$, or \textbf{derivates} of $\{ T \}$.

\begin{thm}\label{thm:2.3} If $T$ is a basic $\mfB$-type and $U$ is a relaxation of $T$, then both $\{ T \}$ and $\{ U \}$ are convex subsets of $\Ray(V)$.
\end{thm}

\begin{proof}
  This is a direct consequence, in fact a reformulation, of the CS-Convexity Lemma \cite[~Lemma ~5.4]{CSFunctions}.
\end{proof}

\begin{examp}\label{exmp:2.4}
  Assume that $S = \{ Y_1, Y_2\}$ with $Y_1 \neq Y_2$. We choose  $\mfB$ as follows.
  \begin{enumerate}\ealph
    \item If $\CS(Y_1,Y_2) \leq e$, $\mfB$ consists of 3 functions, $0$, $\CS(Y_1,-)$, $\CS(Y_2,-)$.
    \item If $\CS(Y_1,Y_2) > e$, $\mfB$ consists of 5 functions, $0$, $\CS(Y_1,-)$, $\CS(Y_2,-)$, $\frac{\CS(Y_1,-)}{\CS(Y_1,Y_2)}$, $\frac{\CS(Y_2,-)}{\CS(Y_1,Y_2)}$.
  \end{enumerate}
From the list of basic types $T$ of CS-profiles on $\overrightarrow{[Y_1,Y_2]}$ in \cite[\S4]{CSFunctions} and their relaxations we infer that
$$ \{ T \} = \Loc_T(Y_1,Y_2), \qquad \{ U \} = \Loc_U(Y_1,Y_2)$$
for a basic $\mfB$-type and a relaxation $U$ of $T$ in the present sense, cf. \cite[Definition 5.5]{CSFunctions}. Thus we have obtained the same partition of $\Ray(V)$ into convex sets as in \cite[\S5]{CSFunctions} and the same relaxation sets.
\end{examp}

Given a finite set $\mfB$ of $S$-basic functions on $\Ray(V)$ we now study the restriction of the $\mfB$-partition to a closed interval $[W,W'].$ If $T$ is a basic $\mfB$-type, let
\begin{equation}\label{eq:2.8}
  [W,W']_T := [W,W'] \cap \{ T \}.
\end{equation}
Provided that this set is not empty, it is called the \textbf{trace of $\{ T \}$ on $[W,W']$}. Since $\Ray(V)$ is the disjoint union of all $\mfB$-strata, the set $[W,W']$ is the disjoint union of all these sets $[W,W']_T$. Assume that $W \in \{ T \}$ and $W' \in \{ T'\}$ where  $T \neq T'$, i.e., $[W,W']$ is not contained in one $\mfB$-stratum $\{T\}$. The traces of $\mfB$-strata on $[W,W']$ are pairwise disjoint convex subsets of $[W,W']$. Since $[W,W']$ is totally ordered with respect to $\leq_W$ (recall that  $X \leq_W Y \Leftrightarrow
[W,X] \subset [W,Y]$), it is trivial that the traces of $\mfB$-strata on $[W,W']$ appear in an ordered sequence
\begin{equation}\label{eq:2.9}
  [W,W'] _{T_0} < [W,W'] _{T_1} < \cdots < [W,W'] _{T_s}.
  \end{equation}
with $T_0 = T$ and $T_s = T'$.
 ($<$ means $<_W$.)

 We pose two questions.
\begin{description}\dispace
  \item [Question A] Are the sets $[W,W']_{T_k}$ intervals, if so of which kind?

 \item[Question B]  If $T'$ is derived from $T$ (i.e., a component of the relaxation $U$ of $T$), how do the sequences $(T_0, \dots, T_s)$ change when we vary $W$ in $\{ T\}$ and $W'$ in $\{ T'\}$?
\end{description}

We first study the case that $\mfB = \{ f_1, f_2\} $ consists of only two $S$-basic functions. (As before we assume that $R$ is root closed.) There exist at most three basic types, ($f_1 < f_2$), ($f_1 = f_2$), ($f_1 > f_2$), which we label $T_0, T_1, T_2$.  Let $[W, W']$ be an interval with $f_1(W) < f_2(W)$, $f_1(W') = f_2(W')$ if there are two strata, and  $f_1(W')  > f_2(W')$ if there are three. $\{T_0\}$ is the stratum containing $W$, while $T_1 = \{f_1 = f_2  \}$, resp. $T_2 = \{f_1 > f_2  \}$ is the stratum containing~ $W'$.

\begin{lem}\label{lem:2.5} We have the following facts.
\begin{enumerate} \ealph
  \item Assume that $f_1(W) < f_2(W)$, $f_1(W') = f_2(W')$, and so $U := (f_1 \leq f_2)$ is a relaxation of $T_0 := (f_1 < f_2)$, with one other component $T_1$. Then there is a unique ray $Z$ in $[W,W']$ such that $f_1 < f_2$ on $[W,Z[$\;, $f_1 =f_2$ on $[Z,W']$. The set $[W,Z[$ is not a closed interval. In other words $[W,Z'] \neq [W,Z[$ for any $Z'$ in $[W,Z[$. It may happen that $Z = W'$ and so $[Z,W'] = \{ W'\} .$ \{We here regard a singleton $\{ X \}$ solely as a closed interval, deviating from thee terminology in \cite[Formula (7.5)]{QF1}, where $\{ X \}$ is also counted as an open and half open interval.\}

  \item Assume that $f_1(W) < f_2(W)$, $f_1(W') > f_2(W')$.
   Then there are two rays $Z_1, Z_2$ in $[W,W']$ such that
   $$ [W,W']_{T_0} = [W,Z_1[ , \quad [W,W']_{T_1} = [Z_1, Z_2] , \quad [W,W']_{T_2} = ]Z_2, W'].$$
   Neither $[W,Z_1 [$ nor  $]Z_2, W']$ is a  closed interval. The basic type $T_1 = (f_1 = f_2)$ is a derivate of both $T_0$ and $T_2$, but, of course, neither $T_0$ is a derivate of $T_2$ nor $T_2$ is a derivate of $T_0$.

\end{enumerate}
  \end{lem}

\begin{proof}
  (a): We infer from Theorem \ref{thm:1.7}.(b) and Remark \ref{rem:1.8}, that the claims in (a) hold in the case that $f_1$ and $f_2$ are monomial functions. Otherwise we have a finite sequence
  $$ A = W_0 < W_1< \cdots < W_{s-1}< W_s = W', $$
  where $<$ means $<_W$, such that both $f_1$ and $f_2$ are monomial on each interval $[W_r, W_{r+1}]$, $0 \leq r \leq s-1$. $U := (f_1 \leq f_2)$ is a relaxation of $T_0 = (f_1 < f_2)$, whence  $\{U\}$ is a convex set. It contains $W'$ and so $[W,W'] \subset \{ U\}$. In particular $f_1(W_k) \leq f_2(W_k)$ for $0 \leq k \leq s$. Thus there exists an index $r \leq s-1$ such that $f_1 < f_2$ on $[W,W_r]$,  $f_1 = f_2$ on $[W_{r+1}, W'].$ Since both $f_1, f_2$ are monomial on $[W_r, W_{r+1}]$, there is, again by Theorem ~ \ref{thm:1.7}.(b), a unique ray $Z \in ]W_r, W_{r+1}[$ such that $f_1 < f_2$ on $[W_r, Z[$, $f_1 = f_2$ on $[Z,W_r]$, and further $[W_r,Z' ] \neq [W_r, Z[$ for $W_r \leq Z' \leq Z$ with respect to $\leq_{W_r}$. Since the ordering $\leq_{W}$ restricts to ~$\leq_{W_r}$ on $[W_r,W']$, it follows that $f_1 < f_2$ on $[W,Z[$, $f_1 = f_2$ on $[Z,W']$, and $[W,Z'] \neq [W,Z[ $ for $Z'$ in $[W,Z]$, proving our claims in part (a).

  \pSkip
  (b): In the case, that $f_1, f_2$ are monomial, our claims in (b) are clear from Theorem \ref{thm:1.7}.(c), there with $[Z_1,Z_2] = \{ Z\}$. The general case now follows by arguing as in the proof of part ~(a).
\end{proof}

\begin{thm}\label{thm:2.6} Let $T$ and $T'$ be different basic types for a set $\mfB = \{ f_1, \dots, f_n\} $
  of $S$-basic functions. Assume that there exist rays $W \in \{ T\}$, $W' \in \{ T'\}$ with
  \begin{equation}\label{eq:2.10}
    [W,W'] \subset \{T\} \cup \{ T' \}.
  \end{equation}
  Then there exists a unique ray $Z \in [W,W']$ such that
  \begin{enumerate}\dispace
    \item $[W,W']_{T} = [W,Z[$, $[W,W']_{T'} = [Z,W']$, or
    \item $[W,W']_{T} = [W,Z]$, $[W,W']_{T'} = ]Z,W']$.
  \end{enumerate}
  In case (1) $T'$ is derived from $T$, while in case (2) $T$ is derived from $T'$. In both cases the set $\{ T \} \cup \{ T'\}$ is convex, and so \eqref{eq:2.10} holds for every pair $(W,W')$ with $W \in \{ T \}$,
  $W' \in \{ T'\}$. $U := T \vee T'$ is a relaxation of $T$ in case (1) and a relaxation of $T'$ in case (2).
\end{thm}

\begin{proof} We have the partition
$$ [W,W'] = [W,W']_T \ \ds{\dot \cup} \ [W,W']_{T'}.$$
  Both sets on the right hand side are closed or half open intervals. This forces that there is a ray $Z$ for which (1) or (2) holds. Suppose we are in case (1). Let $i,j \in \{ 1, \dots, n\} $ with $f_i(W) < f_j(W) $. Then $(f_i < f_j)$ occurs in the formula $T$ and so $f_i < f_j$ on $[W,Z[$ .

  If $(f_i > f_j)$ would occur in the formula $T'$, then $f_i > f_j$ in $]Z,W']$, but this contradicts Lemma \ref{lem:2.5}.(b). Thus  $(f_i < f_j)$ or $(f_i = f_j)$ occurs in $T'$. This proves that $U := T \vee T'$ is a relaxation of $T$, whence $T'$ is a derivate of $T$ and $\{ U  \}  = \{ T \} \cup \{ T'\} $ is convex.

 Case (2) is obtained from Case (1) by interchanging $W$ and $W'$, and we are done also in this case.
\end{proof}

\begin{defn}\label{def:2.7} In the situation of Theorem \ref{thm:2.6} we say that the strata $\{ T\}$ and $\{ T' \}$ are \textbf{neighbors} and that $T'$ is a \textbf{direct derivate} of $T$, respectively $T$ is a direct derivate of~ $T'$.
\end{defn}

\begin{cor}\label{cor:2.8} Two different strata $\{ T\}$ and $\{T' \}$ are neighbors iff the set
$\{ T\} \cup \{T' \}$ is convex.

\end{cor}

\begin{proof}
  This is a immediate consequence of Theorem \ref{thm:2.6}.
\end{proof}

\begin{rem}\label{rem:2.9}
  We state some facts about relaxations of a basic $\mfB$-type $T$, which now are evident.
  \begin{enumerate} \ealph
    \item Given a derivate $T'$ of $T$, there is a unique \textbf{minimal relaxation $U$ of $T$ with component $T'$}, in the sense that $\{ U \} \subset \{ U_1\} $ for every other relaxation $U_1$ of $T$  with component $T'$. It is the conjunction of all formulas $(f_i < f_j)$ and $(f_i = f_j)$ which appear in both $T$ and $T'$, and the formulas $(f_i \leq  f_j)$ such that $(f_i < f_j)$ appears  $T$, but $(f_i = f_j) $ appears in $T'$, $(i,j) \in \{ 1, \dots, n\} $. We denote this minimal relaxation $U$ by $U(T,T')$.

    \item If $U_0$  is a relaxation of $T$ with component $T_1$, and $U_1$ is a relaxation of $T_1$ with component $T_2$, then $U_0 \vee U_1$ is a relaxation of $T$ with component $T_2$.
    \item  If $T_1$ is a derivate of $T$ and $T_2$ is a derivate of $T_1$, then $T_2$ is a derivate of $T$ (but we cannot exclude the possibility  that $U(T,T')$ has more components than $T,T_1,T_2$, cf. Remark  \ref{rem:2.16} below).
        \item
         If $T'$ is a direct derivate of $T$, then
        $$ U(T,T') = T \vee T' .$$

  \end{enumerate}
\end{rem}

We aim for a more intuitive description of the restriction of the $\mfB$-partition to a closed interval $[W,W']$ by using  the terminology of neighbors and direct derivates in Definition \ref{def:2.7}.

\begin{convention}\label{conv:2.10}
  In order not to get distracted by formalities, we use the letters $T, T', \dots $ both for basic $\mfB$-types and the associated sets $\{ T \}, \{ T' \}, \dots $\footnote{This is a habit common in model theory, which we adapt, to denote a definable set and a formula defining it in the same way. }, still calling these sets ``strata''.
\end{convention}

We introduce a sequence of ``\textbf{separation rays}'' as follows. If $[W,W']$ is contained in a stratum, there are no separation rays. Assume now that $W \in T$, $W' \in T'$ for different strata $T$ and $T'$. The strata meeting the set $[W,W']$ appear in a sequence
\begin{equation}\label{eq:2.11}
  T= T_0, T_1, \dots, T_s = T'
\end{equation}
with
\begin{equation}\label{eq:2.11}
[W,W'] \cap T_i \ds{\leq_W} [W,W'] \cap T_{i+1}
\end{equation}
for $0 \leq i \leq s$, cf. \eqref{eq:2.8}. \{We now write $[W,W']\cap T_i$ instead of $[W,W']_{T_i}$.\}
In each set $[W,W'] \cap T_i$, $0 \leq i \leq s$, we pick a ray $W_i$, choosing $W_0 = W$, $W_s = W'$.

The strata $T_i$ and $T_{i+1}$ are neighbors, since $[W_{i-1},W_i] \subset T_{i-1} \cap T_i$. By Theorem \ref{thm:2.6} we have the following alternatives for each $i > 0$.
\begin{description}
  \item[Case 1]
  $$ [W_{i-1},W_{i}] \cap T_{i-1} = [W_{i-1},Z_i[, \qquad [W_{i-1},W_{i}] \cap T_i = [Z_i, W_i].$$
  \item[Case 2]
  $$ [W_{i-1},W_{i}] \cap T_{i-1} = [W_{i-1},Z_i], \qquad [W_{i-1},W_{i}] \cap T_i = ]Z_i, W_i].$$
\end{description}
In case 1 $T_i$ is a direct derivate of $T_{i-1}$, while in Case 2 $T_{i-1}$ is a direct derivate of $T_{i}$.

\begin{defn} We call the rays $Z_0,Z_1, \dots, Z_s$ appearing above  the \textbf{separating rays for~ $\mfB$ on $[W,W']$}.
\end{defn}
If $[W,W']$ meets more than two strata ($s \geq 2$), we can locate the separating rays $Z_i$, $0 \leq i \leq s$, without referring to the case distinction above as follows.

\begin{rem}\label{ref:2.12}
  $Z_0,Z_1, \dots, Z_s$ are the unique rays in $[W,W'] $ with
  $$ ]Z_{i-1}, Z_i[ \ds \subset T_i \cap [W, W'] \ds \subset [Z_{i-1}, Z_i]$$
  for $0 < i < s$.
\end{rem}

Let $f_i, f_j \in \mfB$ be given, $1 \leq i,j \leq n$, $i \neq j$, and assume that $[W,W']$ is not contained in a single  stratum. We thus have the sequence
$ T= T_0 , \dots, T_s = T'$ of strata with $W \in T$, $W' \in T'$, and the sequence of separating rays $Z_0, Z_1, \dots, Z_s$, as discussed above. Recall that $f_i \ds{\Box_{i,j,k}} f_j$ on each $T_k$ with a constant sign $\Box_{i,j,k} \in \{ < , = , > \}$.

We are interested in the distribution of these signs $\Box_{i,j,k}$ on $[W,W']$. We distinguish the cases
\begin{enumerate}
 \ealph
  \item $f_i(W) < f_j(W)$, $f_i(W') = f_j(W')$,
  \item $f_i(W) < f_j(W)$, $f_i(W') > f_j(W')$.
\end{enumerate}
The other possibilities of signs between $f_i(W), f_j(W)$ and  $f_i(W'), f_j(W')$ can be reduced to these cases (a), (b) by interchanging $f_i$ and $f_j$ and/or $W$ and $W'$.

\begin{thm}[Sign Changing Theorem]\label{thm:2.13}
$ $
\begin{enumerate} \ealph
  \item  If $f_i(W) < f_j(W)$ and  $f_i(W') = f_j(W')$, then there is an index $k$, $0 < k \leq s$ such that
  \begin{equation*}\label{eq:I}
    \begin{array}{ll}
      f_i < f_j & \text{on } T_0 \cup \cdots \cup T_{k-1},\\[1mm]
      f_i = f_j & \text{on } T_k \cup \cdots \cup T_{s},\\
    \end{array} 
  \end{equation*}
and
$$
    \begin{array}{ll}
      \{ f_i < f_j \} \cap [W,W'] & = [W,Z_{k-1}[,\\[1mm]
      \{ f_i = f_j \} \cap [W,W'] & = [Z_{k}, W' ].\\
\end{array}
$$

    \item  If $f_i(W) < f_j(W)$ and $f_i(W') >  f_j(W')$, then there are indices  $k, \ell$, $0 < k < \ell \leq s$ such that
  \begin{equation*}\label{eq:II}
    \begin{array}{ll}
      f_i < f_j & \text{on } T_0 \cup \cdots \cup T_{k-1},\\[1mm]
      f_i = f_j & \text{on } T_k \cup \cdots \cup T_{\ell -1},\\[1mm]
      f_i > f_j & \text{on } T_\ell \cup \cdots \cup T_{s},\\
    \end{array} 
  \end{equation*}
and
$$
    \begin{array}{ll}
      \{ f_i < f_j \} \cap [W,W'] & = [W,Z_{k-1}[,\\[1mm]
      \{ f_i = f_j \} \cap [W,W'] & = [Z_{k}, Z_\ell ],\\[1mm]
      \{ f_i > f_j \} \cap [W,W'] & = ]Z_{\ell}, W' ].\\
\end{array}
$$
\end{enumerate}
  \end{thm}
\begin{proof} (a):
  By Lemma \ref{lem:2.5} there is a ray $Z$ in $[W,W']$ such that $f_i < f_j$ on $[W,Z[ $ and $f_i =f_j $ on $[Z,W']$. Since on every set $T_r \cap [W,W']$ we have a constant sign $f_i < f_j$ or $f_i = f_j$, it is evident that $Z= Z_k$. Now all claims in part (a) are evident.

  \pSkip (b):    By Lemma \ref{lem:2.5} we have two rays $Z',Z''$ in $[W,W']$ such that $f_i < f_j$ on $[W,Z'[ $, $f_i =f_j $ on $[W,Z']$, $f_i > f_j $ on $]Z',W']$. This forces $Z' = Z_k$, $Z'' = Z_\ell$, and gives all the claims in part~ (b).
\end{proof}

%
%
%
%

\begin{cor}\label{cor:2.15}
  Given two strata $T \neq T'$ with $T'$ a derivate of $T$, there is a sequence of strata
  $T_0 = T, T_1, \dots, T_s = T'$ such that $T_k$ is a direct derivate of $T_{k-1}, $ for $k=1,\dots,s$
\end{cor}
\begin{proof}
  We choose rays $W \in T$ and $W' \in T'$. Given two different functions $f_i, f_j$ in $\mfB$, we may assume that $f_i(W) < f_j(W),$ perhaps interchanging $f_j$ and $f_j$. Now we are in case (a) of Theorem \ref{thm:2.13}. Either the pair $(f_i,f_j)$ has the same sign, $<$ or $=$, on each $T_k$, $0 \leq k \leq s$, or changes sign from $<$ to $=$ from $T_{k-1}$ to $T_k$ at exactly one $k$, $1 \leq k \leq s$. This test, done with all sets $\{ f_i, f_j\} $, proves that each $T_k$ is a direct derivate of $T_{k-1}$.
\end{proof}
It may  happen that for two different strata $T$ and $T'$ with $T' $ a derivate of $T$ there are several sequences $T_0, \dots, T_s$ of consecutive direct derivates with $T= T_0$, $T' = T_s$. We document this by writing down the chart of all direct derivates in Example \ref{exmp:2.4}.(b) for the ascending basic types listed in \cite[\S7]{QF1}. We mark  a direct derivation $T'$ of a stratum $T$ by an arrow $T \to T'$.

\begin{rem}\label{rem:2.16} Assume that $Y_1, Y_2$ are rays with $\CS(Y_1,Y_2) > e$ and $\mfB$ consists of the five functions in Example \ref{exmp:2.4}.(b). By \cite[Table 4.4]{QF1} we have
  the ascending basic types $A, \partial A, B, \partial B, E, \partial E$, and obtain from the list of relaxations in \cite[Scholium 4.6]{QF1} the chart of direct derivations
%
%
  \begin{equation}
  \label{eq:2.13}
  \begin{gathered}
\xymatrix{   A    \ar @{>}[r]  \ar
@{>}[dr] &  E  \ar @{>}[rd] &   \\
& \partial A  \ar @{>}[r]  &  \partial E  \\
B  \ar @{>}[r]   \ar @{>}[ur] & \partial B  \ar @{>}[ur]    }
\end{gathered}
\end{equation}
  Thus we have two sequences of direct derivations $(A, E , \partial E )$ and $(A, \partial A , \partial E )$
  from $A$ to $\partial E$ and two such sequences $(B, \partial A , \partial E )$ and $(B, \partial B , \partial E )$ from $B $ to $\partial E$.
\end{rem}

\begin{comm}\label{com:2.17} We now have the means at hands to convey  the idea that the $\mfB$-partition of $\Ray(V)$ is a tropical analogue of various classical notions of stratification in differential topology and elsewhere, e.g., real semialgebraic geometry \cite[\S9]{BCR}. We view an interval $[W,W']$ or $[W, W'[$ with its total ordering $\leq _W$ as a ``curve'', where a moving ray $X$ starts at $W$ and travels to $W'$.

Assume that $T'$ is a direct derivate of $T$. For any curve $[W,W']$ with $W \in T$, $W' \in T'$ there is a first ray $Z$ with $Z \in T'$, i.e., $ [W,Z [ \subset T$, $[Z,W'] \subset T'$. This is reminiscent of various curve selection  properties in classical theories. To locate  the strata in the boundary of a given stratum, or special points of them: $[W,Z[$ is a curve in $T$ with a limit ray  $Z$ in~ $T'$. Every direct derivate of $T$ can be located by such a curve $[W,Z[$.


The reader may feel irritated by the primitive shape of this curve selection property. But note that we are in a very special situation. Our basic predicates are formulas $f \ds \Box g$ for two functions $f,g$ where  $\Box$ is one of the signs $<, = , > $, and the functions $f,g$ are $R$-linear combinations of CS-functions, so that each set $\{ X \ds | f(X) \ds \Box g(X)\} $ is convex in $\Ray(V).$

It is natural to study such ``comparison sets'' $\{ X \ds | f(X) \ds \Box g(X)\} $ for functions $f,g$ on suitable convex subsets $C$ of $\Ray(V)$, which are piecewise monomial on each closed interval $[W,W'] \subset C$. As is obvious  from Corollary \ref{cor:1.10}, there is still a partition of each $[W,W']$ into finitely many intervals, on which the pair $(f,g)$ has constant sign. While for $f$ and $g$ in ~$\mfB$ we have at most three such intervals, then there often will be many more, but Questions A and B from above retain their  sense.

It appears to us that for suitable pm-functions a richer stratification theory is available, for which our present theory serves as a starting point and gives a base. Already the class of polynomials $P(f_1, \dots, f_r)$ with $f_1, \dots, f_r \in \mfB$ deserves interest. Also the class of functions $\vrp \circ f $  with $f \in \mfB$ and $\vrp$ a pm-function on $[0,\infty]$ comes to mind, cf. Theorem \ref{thm:1.6}.
\end{comm}

The membership problems addressed at the end of the introduction become more natural and more accessible by putting pm functions into play. Let $\mfM'$ denote the set of all $R$-valued functions on $\Ray(V),$ which restrict to  pm functions on all closed intervals in $\Ray(V)$. Assume as before that $\mfB \subset \mfL$ are sets of linear combinations of CS-functions (but not necessarily finite). Then $\mfB \subset \mfL\subset \mfM\subset \mfM'$. If now $f \in \mfL$ is given, we  ask for a toolbox $(\mfD,J)$, consisting of a set $\mfD \subset \mfM'$ and a (hopefully small) set $J$ of closed intervals in $\Ray(V)$, to test membership of $f\in \mfL$ in $\mfB$ by comparing $f$ with the functions in $\mfD
$ on the intervals in  $J$. The main advantage of the new setting is, that we have much simpler test functions at hands.
We provide a simple instance of such a membership problem.


\begin{example}\label{exmp:3.17}
  Assume that $\CS(Y_1,Y_2) > e$. As before, let $\mfM$ denote the set of linear combinations of CS-functions on $\Ray(V)$ with coefficients in $R \sm \00$. Let $\mfL$ denote the set of functions $\CS(W,-)$ on $[\overrightarrow{Y_1,Y_2}]$ with positive ascending or descending profile, and $\mfB$ denote the set of those $f \in \mfL$ which have a glen, i.e., are of type $B,\partial B,B'$, cf. Example \ref{exmp:2.4}.(b) above (N.B. $\partial B = \partial B'$). Then $(\mfD,J)$ with $J$ consisting of one interval $[Y_1,Y_2]$, and $\mfD $ the set of constant function $g = c > 0$ on  $[Y_1,Y_2]$, is a toolbox for $(\mfB,\mfL)$. Indeed, if $f \in \mfL$ is given, then $f \in \mfB$ iff for the constant function $g$ of value $\min((f(Y_1),f(Y_2))$ we have $f <g$ on some subinterval $]Z_1,Z_2[$ of  $[Y_1,Y_2]$, the glen of $f$.
\end{example}

\section{Entering a direct derivate of a stratum}\label{sec:3}

Returning to a $\mfB$-partition on $\Ray(V)$, we define ``entrance and exit rays'' for a basic $\mfB$-stratum $T$ as follows.

\begin{defn}\label{def:3.1} Assume that $T'$ is a direct derivate of $T$. Given rays $W \in T$, $W' \in T'$ there is a unique ray $Z$ such that $[W, Z[ \; \subset T$, $[Z,W'] \subset T'$. We call such ray $Z$ an \textbf{exit ray} of $T$  and \textbf{entrance ray} of $T'$. We call the set of all entrance rays of $T'$ the \textbf{inner border} of $T'$ and denote this set by $\partial T' $, and the set of all exit rays of $T$ the \textbf{outer border} of  $T$, denoted by $\opr  T$. Thus
\begin{align}
  {\partial} T' & = \{ Z \in T' \ds | \exists W \in T : [W,Z [ \; \cap T' = \emptyset \}, \label{eq:3.1} \\
  \opr  T &  =  \{ Z \in \Ray(V) \setminus T \ds | \exists W \in T : [W,Z [  \subset T \}. \label{eq:3.2}
\end{align}

\end{defn}

Note that each interval $[W,W']$ with $W \in T $, $W' \in T' $ either meets $\partial T$ or $\opr  T$.
Every neighbor of $T$ either contains entrance rays from $T$, or has exit rays in $T$, but not both.
\pSkip

In the following we assume that $T$ and $T'$ are strata with $T'$ a direct derivate of $T$.
\begin{thm}\label{thm:3.2}
  If $W_1,W_2$ are rays in $T$ and $Z_1,Z_2$ are rays in $T'$ such that  $[W_1, Z_1[ \; \subset T$, $[W_2, Z_2[ \; \subset T$, then $[W, Z[ \; \subset T$ and $Z \in T'$ for every $ W \in \; ]W_1, W_2[$ and  $Z \in [Z_1, Z_2].$
\end{thm}
\begin{proof}
  $Z$ is  in $T'$, since $Z_1, Z_2 \in T'$, and $T'$ is convex. Given $X \in [W,Z[$,  we need to verify that $X \in T$. Since $W \in\; ]W_1,W_2[$ , there are vectors $w_1 \in W_1$, $w_2 \in W_2$ such that
  $$W = \ray(w_1 + w_2).$$
  This implies that
  $$\begin{array}{cc}
   X & = \ray((w_1 + w_2 ) + (\lm_1 z_1 + \lm_2 z_2)) \\[1mm]
   & = \ray((w_1 + \lm_1 z_1)+ (w_2  + \lm_2 z_2))
    \end{array}
   $$
  for some $z_1, z_2 \in Z$, $\lm_1, \lm_2 \in R$.  We have
  $$  \ray(w_i + \lm_i z_i) \in [W_i,Z_i] \subset T \qquad \text{for } i= 1,2,$$
  and  conclude that $X \in T.$
  \end{proof}

\pSkip \emph{An illustration  aid.} Imagine that $T$ and $T'$ are neighboring countries and that every closed interval $[W,W'] $   with $W \in T$, $W' \in T' $ is an air path from a location  $W$ in $T$ to a location $W'$ in $T'$. Furthermore, suppose that every plane flying from $T$ to $T'$ is forced to land for control at the first airport $Z$ right after entering $T'$. In this scenario the theorem says that, if $[W_1, Z_1]$ and $[W_2, Z_2]$ represent correct flights from $T$ to $T'$, then for every $W \in ]W_1, W_2[$ and $Z \in [Z_1, Z_2]$ also $[W,Z]$ represent a correct flight.

\begin{cor}\label{cor:3.3} $ $
\begin{enumerate} \ealph
  \item Let  $W \in T$ and $Z_1,Z_2\in T'$ with $[W, Z_1[ \; \subset T'$, $[W, Z_2[ \; \subset T'$. then $[W,Z[ \; \subset T$ for every $Z \in [Z_1, Z_2]$.
  \item    Let  $W_1,W_2 \in T$  and $Z \in T'$, and assume that $[W_1, Z[ \; \subset T$, $[W_2, Z[ \; \subset T$. Then $[W, Z[ \; \subset T$  for every $ W \in [W_1, W_2]$.

\end{enumerate}
\end{cor}
\begin{proof} (a): Apply Theorem \ref{thm:3.2} with $Z_1= Z_2 = Z$. \pSkip
(b):  The claim holds for $W$ in $]W_1, W_2[  $ by  Theorem \ref{thm:3.2}.  For the remaining two rays $W_1$ and ~$W_2$ it holds by assumption.
\end{proof}

\begin{thm}\label{thm:3.4}
   Given $W_1,W_2 \in T$ and $Z_1,Z_2\in T'$, assume that all sets
    $[W_1,Z_1[$, $[W_1,Z_2[$, $[W_2,Z_1[$, $[W_2,Z_2[$ are contained in $T$.
    Then for every $W \in  [W_1, W_2]$ and $Z \in [Z_1, Z_2]$ again
      $[W, Z[ \subset T$ and $Z \in T'$.
\end{thm}
\begin{proof}
  $Z$ is in $T'$ since $T'$ is convex.
   If $W \in\; ]W_1,W_2[$ and $Z \in [Z_1,Z_2]$, then $[W,Z [ \subset T$ by Theorem \ref{thm:3.2}. If
   $W =W_1$ and $Z \in [Z_1,Z_2],$ then $[W,Z [ \subset T$ by Corollary  \ref{cor:3.3}. If $W =W_2$, we obtain the same by interchanging $W_1$ and $W_2$.
\end{proof}

Encouraged by Theorem \ref{thm:3.2} we search systematically for convex subset $A \subset T$, $B \subset T'$ with $[W, Z[ \in T$ for all $W \in A$, $Z \in B$. (In short, all flights from $A$ to $B$ are correct.)
Note that then $B$ is a  convex subset of $\opr T \cap T'$.

\begin{defn}\label{def:3.5} (As before $T'$ is a direct derivate of $T$.)
\begin{enumerate}\ealph
  \item  For a given $W \in T$ we define
  $$\sphericalangle (W,T') := \{ Z \in T' \ds| [W,Z[ \subset T\}.  $$
  \item  For a given $Z \in T'$ we define
  $$\sphericalangle (T,Z) := \{ W \in T \ds| [W,Z[ \subset T\}.  $$
\end{enumerate}
\end{defn}
Both of these sets are convex by Corollary \ref{cor:3.3}. In illustrative  terms:
$\sphericalangle (W,T')$ represents the set of airports in $T'$ admitted for $W$, and $\sphericalangle (T,Z)$ represents the set of airports in $T$ for reaching $Z$.

It is obvious from Definition \ref{def:3.1} that $(\opr T) \cap T'$ is the union of all sets $\sphericalangle (W,T')$ with $W$ running through  $T'$,
\begin{equation}\label{eq:3.3}
  (\opr T) \cap T' = \bigcup_{W \in T} \sphericalangle (W,T').
\end{equation}

\begin{defn}\label{def:3.6} $ $
\begin{enumerate}\ealph

\item A \textbf{junction} for $T$ and $T'$ is a triple $(W, W' ,Z)$ with $Z \in T'$, $[W,Z [ \subset T$, and $[W',Z[ \subset T$. In other terms $Z \in T'$ and $W,W' \in \sphericalangle(T,Z)$.

    \item A \textbf{butterfly}  (for $T$ and $T'$) is a quadruple $(W,W',Z,Z')$ with
    $Z,Z' \in T'$, $Z \neq Z'$, and $[W,Z [$, $[W',Z [ $, $[W,Z'[$, $[W', Z'[$ subsets of $T$.
\end{enumerate}
\end{defn}
Thus a butterfly consists of two junctions $(W,W',Z)$, $(W,W',Z')$ which have a common ``base'' $[W,W']$.

$$
\xymatrixrowsep{5mm}
\xymatrixcolsep{6mm}
\xymatrix{   \ar @{-}[ddrrrrr] & Z \ar @{-}[ldd] \ar @{-}[rrddd] & & \\
& & & & Z' \ar @{-}[ddl] \ar @{-}[dllll] \\
W & & & & & T' \\
& & & W'   }
$$

\begin{rem}
  As consequence of Theorem \ref{thm:3.4} we have the following fact. If $(W_1,W_2,Z_1,Z_2)$ is a  butterfly  for $T$ and $T'$, then, for any rays $W'_1, W'_2$ in $[W_1,W_2]$ and  $Z'_1, Z'_2$ in $[Z_1, Z_2]$, $(W_1',W_2',Z_1',Z_2')$ is again a butterfly for $T$ and $T'$. In particular $(W,W_1,Z_1,Z_2)$ is a butterfly for any $W \in [W_1, W_2],$ implying that $[Z_1,Z_2] \subset \opr T \cap T'$, a result stated already in Corollary ~\ref{cor:3.3}.(a).
\end{rem}

Searching for explicit instances where junctions and butterflies occur, we need mild regularity properties of the involved CS-functions. They pertain to the bilinear companion $b: V \times V \to R$.

\begin{defn}\label{def:4.8} Given an $S$-basic CS-function $f = \sum_{j=1}^n \gm_i\CS(Y_j, - )$ with $\gm_j \neq 0$ for all $j$ (cf. Definition \ref{def:3.1}), we call  a ray $W$ in $V$ \textbf{$f$-regular}, if $\CS(Y_j, W) >0$ for all $j \in \{ 1, \dots, n\}$.
\end{defn}
\noi 
 Note that $\CS(Y_j,W) > 0$ iff $b(y_j, w)>0$ for $y_i \in Y_j$, $w \in W$.

\begin{lem}\label{lem:4.9}
  Assume that $W$ and $W'$ are rays in $T$ and $Z$ is a ray in $T'$ with $[W,Z[ \subset  T$. Choose vectors $w \in W$, $w' \in W'$, $z \in Z$, for which $W = \ray(w)$, $W' = \ray(w')$, $Z = \ray(z)$. Assume that $f$ is an $S$-basic CS-function, for which the ray $Z$ is $f$-regular.  Then there exist values $d>0$, $c>0$ in $R$ (which can be made explicit) such that $f$ is constant on every interval $[\ray(z + \mu w),\ray(z + \mu w + \lm w')]$, where $0 \leq \lm \leq c$, $0 \leq \mu \leq d$.\end{lem}

\begin{proof}
  It suffices to verify the assertion  for $f = \CS(Y,-)$  with $Y\in S$, $\CS(Y,Z) >0$. Choosing $y\in Y$, we have $b(y,z)>0$ and
  $$ \CS(z + \mu w + \lm w', y) = \frac {b(z + \mu w + \lm w',y)^2}{q(z + \mu w + \lm w')q(y)}. $$
Since $q(z) >0$, there is some $a>0$ such that $q(z + \mu w + \lm w') = q(z)$ for $\lm + \mu \leq a$. Since $b(z,y) >0$, we then find $c >0$, $d >0$ such that
$$b(z + \mu w + \lm w',y) = b(z,y) + \mu b(w,y) + \lm b(w',y) = b(z,y)$$
for $\lm \leq c$, $\mu \leq d$. These bounds $a,c,d$ can be easily made explicit in terms of $z,w,w'$.
\end{proof}

For a given basic set $S = \{ Y_1, \dots, Y_n\} $ of rays we choose vectors $y_i \in Y_i$, $i \in \{ 1, \dots, n\} $.

\begin{defn}\label{def:4.10} We say that a vector $u \in V$ is \textbf{$S$-regular}, if $b(u,y_i) > 0$ for every $i \in \nSet$.
\end{defn}
$S$-regular vectors occur frequently as we explain briefly.
\begin{lem}\label{lem:4.11}
  Assume that the companion $b$ of $q$ has zero radical on the submodule $\tlV$ of $V$ corresponding to the convex subset $T \cup T'$ of $\Ray(V)$, cf. \cite[\S2]{CSFunctions}. This means that for every nonzero vector $x \in V$ with $\ray(x) \in T \cup T'$, there is some nonzero $v \in V$ with $\ray(v) \in T \cup T'$ and $b(x,v) >0$.
\end{lem}
\begin{proof}
  Our goal is to find a vector $u\in V$ for which  $b(y_i,u) >0$ for $i \in \{1, \dots, n \}$  with  $\ray(u) \in T'$. We choose vectors $v_i \in V$ with $b(y_i,v_i) > 0 $ for $i \in \{1, \dots, n \}$, which is possible by our assumption. We obtain for $v:= v_1+ \cdots + v_n$ that $b(y_i,v) >0$ for $i \in \{1, \dots, n \}$.

   If $v \in T'$, we are done by taking $u =v$.
  Assume that $\ray(v) \in T$. We then pick some $x \in V$ with $\ray(x) \in T'$. Let $Z$ denote the entrance ray of $[\ray(v),\ray(x)]$ in $T'$. We have a vector $z \in Z$ with $z = v + \zt x$ and nonzero $\zt \in R$. Then $b(z,y_i) >0$ for $i \in \{1, \dots, n \}$ and $\ray(z) \in T'$. We are done by taking $u =z$.
\end{proof}

\begin{example}\label{exmp:4.12}
  Assume that the pair $(q,b)$ is balanced \cite[\S1]{QF1}. Then $b(x,x) = q(x) > 0$ for every $x \neq 0$ in $V$, and so the restriction of $b$ to any nonzero submodule of $V$ has zero radical.
\end{example}
We are ready for a first positive result about existence of butterflies.

\begin{thm}\label{thm:4.13}
  Assume that $T'$ contains an $S$-regular ray $U$. Given rays $W,W'$ in $T$, let $Z \in T'$ denote the entrance ray of $[W,U]$ in $T'$. Choose vectors $w \in W$, $w' \in W'$ and $z\in Z$. Then there exists some $c > 0 $ in $R$ such that $\ray(z + \lm w')$ is the entrance ray in $T'$ of the interval
  $[\ray(w +  \lm w'), \ray(z +  \lm w')]$   for every $\lm \in [0,c]$.
\end{thm}
\begin{proof}
  By Lemma \ref{lem:4.11} there exist elements $c > 0$, $d > 0$ in $R$ such that  every $f \in \mfB$ is constant on every interval  $[\ray(z + \mu w ), \ray(z +  \mu w + \lm w')]$ with $0 \leq \lm \leq c$, $0 \leq \mu \leq d$. Given $f_1, f_2 \in \mfB$ where  $f_1 < f_2$ on $T$ and $f_1 = f_2$ on $T'$, we conclude for all $\lm \in [0,c]$, $\mu \in \; ]0,d]$, that
  $$f_1(z +  \mu w + \lm w') = f_1(z + \mu w) < f_2 (z + \mu w) =f_2(z +  \mu w + \lm w'),$$
  while
  $$ f_1(z + \lm w') = f_1(z) = f_2(z) = f_2(z + \lm w'). $$
  This makes it evident that $\ray(z + \lm w')$ is the entrance ray in $T'$ of the interval  $[\ray(w +  \lm w'), \ray(z +  \lm w')]$.
\end{proof}

\begin{cor}\label{cor:4.14} With the notation of Theorem \ref{thm:4.13}, $ W_1 := \ray(w + cw') \in T$, $Z_1 := \ray(z + cw') \in T'$. Then the quadruple $(W,W_1, Z, Z_1)$ is a butterfly  for $T$ and $T'$.

\end{cor}
\begin{proof}
  a) We know that $Z$ and $Z_1$ are rays in $T'$.
  \pSkip
  b) Let $X \in [W, Z_1[$. Then  $X = \ray(z + cw' + \mu w)$, for some $\mu >0$.  If $0 < \mu \leq d$, then $x \in T$. It follows that $X \in T$  for all $X \in [W, Z_1[$, since $[W, Z_1[$ is convex and $W \in T$.
  \pSkip
  c) Let $X \in [W_1, Z[$. Now $X = \ray(z+ \mu(w + c w')) = \ray(z+ \mu w + \mu c w')$ for some $\mu >0 $. If $\mu \leq \min(1,d)$, then $X \in T$. It follows that this holds for all $X \in [W_1, Z[$, since $W_1 \in T$ and $[W_1,Z[$ is convex.
\end{proof}

Given rays $W, W' \in T$ and $U \in T'$ we now describe a process which sometimes gives us a junction $(W,W',Z)$, i.e., a ray $Z \in T'$ with $[W,Z[ \subset T $, $[W', Z] \subset T$.

\begin{construction}\label{const:4.15} We start with the entrance ray $Z_0$ of $[W,U]$ (in $T'$), i.e., the ray $Z_0$ in $[W,U]$ with  $[W,Z_0[ \subset T $, $Z_0 \in T'$, and then obtain a sequence $(Z_0, Z_1,Z_2, \dots)$ of rays in $(\opr T) \cap T'$ as follows \\
\quad $Z_1:= $ the entrance ray of $[W',Z_0],$\\
\quad $Z_2:= $ the entrance ray of $[W,Z_1]$,\\
\quad $Z_3:= $ the entrance ray of $[W',Z_2],$\\
\quad etc. \\
  We meet one of the following cases:
  \begin{description}
 \item[Case 1] There exists a first value $k \in \N_0$ with $Z_k = Z_{k+1}.$ Then the sequence ``stops'' at~ $k$, i.e., $Z_k = Z_{k+1} = \cdots$. This final ray $Z = Z_k$ gives  us a junction $(W,W',Z)$, i.e., a ray $Z \in T'$ with $[W,Z[ \subset T $, $[W', Z[ \subset T$.
 \item[Case 2] We have $Z_k \neq Z_{k+1}$ for every $k$. We then say that the sequence $(Z_0, Z_1,Z_2, \dots)$ is a \textbf{gorge} in $(\opr T) \cap T'$.

     \end{description}
\end{construction}

     Case 2 reflects a tragical scenario. ``Bob'', situated at $W \in T$,  wants to meet ``Alice'', situated at $W' \in T$, in $T'$ by legal   flights of both. He proposes the  airport $Z_0$. But Alice realizes that $Z_0$ is not a legal entrance for her, and thus proposes the legal entrance $Z_1$ on the air path $[W', Z_0]$, etc. Bob and Alice will never come together in $T'$.

  We call the process described in Construction  \ref{const:4.15}  the \textbf{junction process for} $(W,W', U)$, although this process can give us a gorge instead of a junction.

   We obtain more insight into this process by pursuing it on the level of vectors instead of rays. First notice that the process coincides with the junction process for $(W,W', Z_0)$, where ~$Z_0$ as  is above the entrance ray for $[W,U]$ in $T'$.

\begin{schol}
  Retaining the notations in Construction \ref{const:4.15}, we choose vectors $w \in W$, $w' \in W'$,
  $z_0 \in Z_0$. Since $Z_1 \in [W', Z_0]$ there is a scalar
  $\lm_1 \in R$ such that $z_1= z_0 + \lm_1 w' \in Z_1$.
  We note in passing that $\lm_1$ is uniquely determined by the vector $\lm_1 w'$, since $\lm_1^2 q(w') = q(\lm_1 w)$ and $q(w') \neq 0$, and $q(w')$ is a unit in the semifield $R$. Furthermore, there is a scalar $\lm_2 \in R$ for which
  $$z_2 = z_1 + \lm_2 w = z_0 + \lm_2 w + \lm_1 w' \in Z_3,$$
  and a scalar $\lm_3 \in R$ for which
  $$z_3 = z_3 + \lm_3 w'  = z_0 + \lm_2 w + (\lm_1 + \lm_3) w' \in Z_3,$$
  where again $\lm_2$ and $\lm_3$ are units determined by the vectors $\lm_2 w$ and $\lm_3 w'$.

  Continuing in this way  we obtain a sequence $(z_0, z_1, z_2, \dots )$ of vectors $z_i \in Z_i$  and a sequence $(\lm_0, \lm_1, \lm_2, \dots )$ of scalars, starting with $\lm_0 = 0$, such that for every $i \in \N_0$,
  \begin{equation}\label{eq:4.6}
    z_{2i + 1} = z_0 + (\lm_{2i} +\lm_{2i-2} + \cdots  ) w +  (\lm_{2i+1} +\lm_{2i-1} + \cdots  )w',
  \end{equation}
  and
  \begin{equation}\label{eq:4.6}
    z_{2i} = z_0 + (\lm_{2i} +\lm_{2i-2} + \cdots  ) w +  (\lm_{2i-1} +\lm_{2i-3} + \cdots  ) w'.
  \end{equation}
  Recalling  that any (finite) sum of scalars is the maximum of these scalars, we infer that our process stops iff $\lm_{k+2} \leq \lm_k$ for some $k \in \N_0$. In this case $Z_{k+2} = Z_{k+1}$, and we have produced a junction $(W,W',Z)$ for $T$ and $T'$
 with $Z:= Z_{k+1} = Z_{k+2}$.  \end{schol}

We briefly  sketch a case where it makes sense to associate to the gorge a ``limit ray'' $Z_\infty \in T \cup T'$.
Assume that $R = (\RR_{\geq 0}, +, \cdot \;)$, the classical max-plus semiring in multiplicative notation. Then both sequences $\lm_0 = 0 < \lm_2 < \lm_4 < \cdots $  and
$\lm_1 < \lm_3 < \lm_5 < \cdots $ have suprema
$$ \sig = \sum_{u} \lm_{2i},
\qquad \tau = \sum_{i+1} \lm_{2i +1}$$ in the ordered monoid $\RR_{ \geq 0}\cup \{ \infty\}$,  and we obtain a vector
$$ z_\infty := z_0 + \sig w + \tau w'$$
with associated ray $Z_\infty = \ray(z_\infty)$.

\begin{prop}\label{prop:4.17}
  Both sets $[ W, Z_\infty[ $ and $[ W', Z_\infty[ $ are contained in $T$. Thus, if $Z_\infty \in T'$, then
  $(W,W',Z_\infty)$ is a junction.
\end{prop}
\begin{proof}
  Since $\sig = \lm_2 + \sig$, $\tau =\lm_1 + \tau$, we obtain for any $\al \in R \sm \00$
  $$
  \begin{array}{ll}
    w + \al z_\infty & = w + \al(z_0 + \sig w + \tau w') \\
    & = w + \al(z_0 + \lm_2 w + \sig w + \tau w') \\
    & = w + \al z_2 + \al \sig w + \al \tau w', \\
  \end{array}
  $$
  and
  $$
  \begin{array}{ll}
    w' + \al z_\infty & = w' + \al(z_0 + \sig w + \lm_1 w' + \tau w') \\
 & = w' + \al z_1 + \al \sig w +  \al \tau w'. \\\end{array}
  $$
  Since $\ray(w + \al z_2) \in [W,Z_2[ \subset T$, $\ray(w' + \al z_1) \in [W',Z_1[ \subset T$
  for any $\al \in R \sm \00$, we conclude that
  $\ray(w + \al z_\infty) \in [W,Z_\infty [ \subset T$, $\ray(w' + \al z_\infty) \in [W',Z_\infty [ \subset T$, as desired.
\end{proof}

\begin{defn}\label{def:4.18} Given nonempty sets $U \subset T$ and $P \subset T'$, we define  new
sets
$$L(U) := \bigcap_{W \in U} \sphericalangle (W,T'), \qquad
S(P) := \bigcap_{Z \in P} \sphericalangle (T,Z) . $$
\end{defn}
In other words, $L(U)$ is the set of rays $Z$ which can serve as correct entrances in $T'$ for every $W \in U$, and $S(P)$ is the set of starting points $W$ for a correct flight with entrance at every $Z \in P$. Note that these sets $L(U)$ and $S(P)$ are convex, but may be empty.

Our constructions above (Theorem \ref{thm:4.13}, Corollary \ref{cor:4.14}, and to some extent Proposition ~\ref{prop:4.17}) give us instances of sets $U \subset T$ with $L(U)$ not empty and of sets $P \subset T$ with $S(P)$ not empty. Such sets $U$ and $P$ can be enlarged by a widely used formal   saturation process.

\begin{thm}\label{thm:4.19} $ $
  \begin{enumerate} \ealph
    \item If $U \subset T$ and $L(U) \neq \emptyset$, then $LSL(U) = L(U)$, and $SL(U)$ contains every set $U_1 \supset U$ with $L(U_1) = L(U)$.

    \item If $P \subset T'$ and $S(P) \neq \emptyset$, then $SLS(P) = L(P)$, and $LS(P)$ contains every set $P_1 \supset P$ with $S(P_1) = S(P)$.

  \end{enumerate}
\end{thm}

\begin{proof}
  By definition $SL(U) \supset U$ and $LS(P) \supset P$. Thus $SL(U) \neq \emptyset$, $LS(P) \neq \emptyset$. We can apply the operators $L$ and $S$ to $SL(U)$ and $LS(P)$ respectively to obtain

  $$\begin{array}{cc}
      LSL(U) &  = L(SL(U)) \subset L(U), \\[1mm]
      SLS(P) & = S(LS(P)) \subset S(P), \\
    \end{array}
$$
but also have
$$\begin{array}{cc}
      LSL(U) & = LS(L(U)) \supset L(U), \\[1mm]
       SLS(P) & = SL(S(P)) \supset S(P) . \\
    \end{array}
$$
Thus $LSL(U) = L(U)$ and $SLS(P) = S(P).$
The second claims in (a)  and (b)  are now evident, since  the operators $LS$ and $SL$ enlarge nonempty sets in $T'$ and $T$ respectively.
\end{proof}

\section{Approaching isotropic rays}\label{sec:4}

We abandon the overall assumption in \S\ref{sec:2} and \S\ref{sec:3} that the quadratic form $q$ on the $R$-module $V$ is anisotropic (but retain the assumption that $R = eR$ is a root closed semifield). We utilize  the study in \S\ref{sec:2} of the ray space of the $R$-submodule $V_\an$ of anisotropic vectors in ~ $V$. Thus, now $S = \{ Y_1, \dots, Y_n\} $ is a finite subset of $\Ray(V_\an)$ and $\mfB = \{f_1, \dots, f_n \}$ is a finite set of $S$-basic CS-functions. Since $V + V_\an = V_\an$, the set $\Ray(V_\an)$ is convex in $\Ray(V)$ and furthermore
$$ ]W, W' ] \subset \Ray(V_\an)$$
for any isotropic ray $W$ (i.e., $q(W)= \{ 0 \} $)
and anisotropic ray $W'$. Moreover, if $W, W'$ are isotropic rays, but $]W, W'[$ contains some anisotropic  ray then
$$ ]W, W' [ \ds\subset \Ray(V_\an).$$
It turns out that the Sign Changing Theorem \ref{thm:2.13} remains valid on $]W,W']$ and $]W,W'[$
respectively with the following modifications. As customary, we write
 $q(W) = \{ q(w) \ds | w ~\in~ W\} $.

\begin{thm}\label{thm:4.1} $ $
  \begin{enumerate} \ealph
    \item
    If $q(W) = \00$ and $q(W') \neq \00$, then there is a unique sequence of strata $T_1, \dots, T_s$
in $\Ray(V_\an)$ which meets $]W,W']$, such that $s \geq 1$ and, with respect to  $\leq_W$:
    $$ ]W, W']_{T_0} < \ ]W, W']_{T_1} < \cdots < \ ]W, W']_{T_s},$$
    where $]W, W']_{T_k}: = \, ]W, W'] \cap {T_k}$. (N.B. All these sets are convex.) The strata $T_{k-1}$ and $T_k$ are neighbors for $1 \leq k \leq s$. If $\tlW \in \; ]W, W']_{T_0}$, then the interval  $[\tlW, W']$ meets all these strata, and the sequence of separating rays $Z_0, Z_1, \dots, Z_s$ of $[\tlW,W']$ is independent of the choice of $\tlW$ in $]W, W']_{T_0}$. If $f_i,f_j \in \mfB$ are given with
    $$ f_i(\tlW) < f_j(\tlW), \qquad f_i(W') = f_j(W') $$
or
    $$ f_i(\tlW) < f_j(\tlW), \qquad f_i(W') > f_j(W'), $$
then  Theorem \ref{thm:2.13}.(a) remains valid  with $[W,W']$ replaced by $]W,W']$, and
$[W,Z_{k-1}]$ replaced by  $]W, Z_{k-1}]$.

\item Assume  that both $W$ and $W'$ are isotropic rays in $V$, and that the interval $[W,W'] $ contains anisotropic rays. Then $]W,W'[ \; \subset \Ray(V_\an)$. Let $T_0, T_1, \dots, T_s$ denote the sequence of strata that meet $]W,W'[$, such that with respect to $\leq_W$:
    $$ ]W, W'[_{T_0} \ds < ]W, W'[_{T_1}\ds < \cdots\ds  < ]W, W'[_{T_s},$$
   where $]W, W'[_{T_k}: = \; ]W, W'[ \cap {T_k}$. (N.B. Again these strata are convex.) The strata $T_{k-1}$ and $T_k$ are neighbors for $1 \leq k \leq s$.
   Given rays $\tlW$  in $]W, W'[_{T_0}$ and $\tlW'$  in $]W, W'[_{T_s}$, the interval  $[\tlW, \tlW']$ meets all  strata $T_0, T_1, \dots, T_s$, and the sequence of separating rays $Z_0, Z_1, \dots, Z_s$ of $[\tlW,\tlW']$ is independent of the choice of $\tlW$ and $\tlW'$. If $f_i,f_j \in \mfB$ are given with
    $$ f_i(\tlW) < f_j(\tlW), \qquad f_i(\tlW') = f_j(\tlW') $$
or
    $$ f_i(\tlW) < f_j(\tlW), \qquad f_i(\tlW') > f_j(\tlW'), $$
then  Theorem \ref{thm:2.13}.(b) remains valid  with $[W,W']$ replaced by $]W,W'[$,
$[W,Z_{k-1}[$ replaced by  $]W, Z_{k-1}[$, and
$[Z_\ell,W']$ replaced by  $[Z_{\ell},W'[$.
  \end{enumerate}
\end{thm}
\begin{proof} Just observe that for the sign $\Box \in \{ < , =, >\} $ we have $f_i (\tlW) \Box f_j(\tlW)$ for some $i \neq j$ in $\{1, \dots, n \}$ iff the formula $f_i \Box f_j$ occurs in the conjunction  $T_0$ and, if $W'$ is isotropic, $f_i(\tlW') \Box f_j(\tlW')$  holds iff
 the formula $f_i \Box f_j$ occurs in  the conjunction  $T_s$. Then apply Theorem \ref{thm:2.13} to $[\tlW, W']$  and $[\tlW, \tlW']$ respectively.
\end{proof}

The problem  arises, to determine the ``entrance stratum'' $T_0$ of $]W,W']$ in Theorem ~ \ref{thm:3.2}.(a), and to get a hold on the rays $\tlW$ in $]W,W']$ such that $]W,\tlW] \subset T_0$.

\begin{problem} $ $
  \begin{enumerate} \ealph
    \item Given an isotropic vector $\veps \neq 0$ and a vector $\eta $ in $V$ with $q(\veps + \eta) \neq 0$, find the first stratum $T_0$ met by $]W,W']$ for $W = \ray(\veps)$, $W' = \ray(\veps+ \eta)$. How much does  $T_0$ depend   on the choice of $\veps$ and $\eta$?
    \item Find for given $\veps$ and $\eta$ an explicit bound $t_0 >0 $ such that $\ray(\veps+ t \eta)$ stays in $T_0$ for $0 < t \leq t_0$.
  \end{enumerate}
\end{problem}

We pursue this problem in the case of Example \ref{exmp:2.4}, where a stratum $T$ is determined by the CS-profile of any ray $W \in T$ on an oriented interval $\overrightarrow{[Y_1,Y_2]}$.
Since we will refer to the computations in \cite[\S3]{CSFunctions}, we relabel $\veps = \veps_1$, $Y_2 = \ray(\veps_2)$, $Y_3 = \ray(\veps_3)$ and abbreviate $\al_{12} = b(\veps_1, \veps_2)$,
$\al_{13} = b(\veps_1, \veps_3)$. We analyze the profile of the function
$$ f_t(\lm)= \CS(\veps_1 + t \eta, \veps_2+\lm \veps_3)$$
with $\lm$ running in $[0,\infty]$ for small $t >0$.
This make sense if either $b(\veps_1,\eta) > 0$ or  $b(\veps_1,\eta) = 0$, $q(\eta) > 0$, since
$ q(\veps_1 + t \eta ) = t b(\veps_1,\eta) + t^2 q(\eta).$
Then
\begin{equation}\label{eq:4.1}
  f_t(\lm) = \frac{\al_{12}^2 + \lm^2 \al_{13}^2+t^2[b(\eta,\veps_2)^2 + \lm^2b(\eta,\veps_3)^2]}
  {q(\veps_2 + \lm\veps_3)(tb(\veps_1,\eta)+t^2q(\eta))}.
\end{equation}
The word ``profile'' here means the monotonic behavior of the function $\lm \mapsto f_t(\lm)$ on $[0,\infty]$, as in \cite[\S3]{CSFunctions}. Thus the profile of $f_t(\lm)$ does not change if we omit the nonzero constant factor  $tb(\veps_1,\eta)+t^2q(\eta)$ in \eqref{eq:4.1}.

We distinguish several cases:
\begin{enumerate} \dispace
  \item[A.]  $\al_{12}>0$, $\al_{13} >0$.
  If $tb(\eta,\veps_2) \leq \al_{12}$ and $tb(\eta,\veps_3) \leq \al_{13}$, then $f_t(\lm)$ has the same profile as
  $$  \frac{\al_{12}^2 + \lm^2 \al_{13}^2}{q(\veps_2 + \lm \veps_3)}.$$
  This happens if
  \begin{equation}\label{eq:4.2}
  0 < t \leq \min\bigg( \frac{\al_{12}}{b(\eta,\veps_2)}, \frac{\al_{13}}{b(\eta,\veps_3)} \bigg).
  \end{equation}
  (If $b(\eta, \veps_i) = 0$, read $\frac{\al_{1i}}{b(\eta,\veps_i)} = \infty$.)

  \item[B.]  $\al_{12} = \al_{13} = 0$.
  For every $t>0$ $f_t(\lm)$ has the same profile as
  $$  \frac{b(\eta, \veps_2)^2 + \lm^2 b(\eta,\veps_3)^2}{q(\veps_2 + \lm \veps_3)}.$$
  In particular,  if $b(\eta, \veps_2) = b(\eta, \veps_3) = 0$, then $T_0$ is the stratum containing $\ray(\eta)$.

  \item[C.]  $\al_{12}>0$, $\al_{13} = 0$, whence
  $$ f_t(\lm) = \frac{\al_{12}^2 +t^2[b(\eta,\veps_2)^2 + \lm^2b(\eta,\veps_3)^2]}
  {q(\veps_2 + \lm\veps_3)}.$$

  \item[C1.]  If $b(\eta, \veps_3) = 0$, then $f_t(\lm)$ has for every $t>0$ the same profile as
  $$ \frac{1}
  {q(\veps_2 + \lm\veps_3)}.$$
Thus $\ray(\veps_1 + t\eta)$ stays in a fixed stratum for all $t>0$, which is independent of $\veps_1$ and $\eta$ (as long as $b(\eta,\veps_3) = 0$, $b(\veps_1,\veps_2) > 0$, $b(\veps_1,\veps_3) = 0$, $q(\veps_1) = 0$).

\item[C2.]  Assume that  $b(\eta, \veps_3) > 0 $ and, as before that  $\al_{12} > 0$, $\al_{13} =0$. We now resort to the list of basic types in \cite[\S4]{CSFunctions}.
Let $\al_2 = q(\veps_2)$, $\al_3 = q(\veps_3)$. We study the CS-profile of $\zt := \veps_1 + t \eta$ on $\overrightarrow{[Y_1,Y_2]}$. We compute
\begin{equation}\label{eq:4.3}
  \CS(\zt,\veps_2) = \frac{\al_{12}^2}{\al_2q(\zt)}, \qquad \CS(\zt,\veps_3) = \frac{t^2 b(\eta, \veps_3)^2}{\al_3q(\zt)},
\end{equation}
Assume first that $\CS(\veps_2,\veps_3) > e$. It is evident from \cite[Table 4.3]{CSFunctions} and \cite[Scholium ~4.5]{CSFunctions}
that the ray of $\veps_1 + t \eta$ has a CS-profile of type $A'$ for small $t$. More precisely this happens iff
$$ \CS(\zt,\veps_3) < \frac{\CS(\zt,\veps_2)}{\CS(\veps_2,\veps_3)},$$
which by \eqref{eq:4.3} means that
$$ t^2 < \frac{\al_{12}^2 \al_2}{\al_3b(\eta, \veps_3)^2} \cdot \frac{\al_2 \al_3}{\al_{23}^2},$$
i.e.,
\begin{equation}\label{eq:4.4}
  t < \frac{\al_{12} \al_2}{\al_{23}b(\eta, \veps_3)}.
\end{equation}
Assume  now that $\CS(\veps_2, \veps_3) \leq  e$.
The ray of $\veps_1 + t \eta$ has a CS-profile of type $C'$ for small $t$. More precisely this happens if
$$ \CS(\zt,\veps_3) < {\CS(\zt,\veps_2)},$$
which by \eqref{eq:4.3} means that
$$ t^2 \frac{b(\eta,\veps_3)^2}{\al_3q(\zt)} < \frac{\al_{12}^2 }{\al_{2} q(\zt)},$$
equivalently
\begin{equation}\label{eq:4.5}
  t < \frac{\al_{12}}{b(\eta, \veps_3)}\sqrt{\frac{\al_3}{\al_2}}.
\end{equation}
\end{enumerate}
We write down two consequences of this  analysis of profiles.

\begin{prop}\label{prop:4.3} Let $\eta \in V$, $q(\eta) > 0$. For each  isotropic vector $\veps \neq 0$ with $b(\veps, \veps_2) = b(\veps, \veps_3) =0$,  the ray of $\veps + t \eta$ has the same CS-profile as $\ray(\eta)$ for every small $t>0$, and thus lies in the same stratum as $\ray(\eta)$.
\end{prop}
\begin{proof}
  cf. Case B above.
\end{proof}
\begin{prop}\label{prop:4.4} If $b(\veps, \veps_2) >0$, $b(\veps, \veps_3) = 0$, then for each $t >0$ and all $\eta$ with $q(\eta) > 0$, $b(\eta, \veps_3) =0$, the ray of  $\veps+ t  \eta$ is  contained in a fixed stratum $T$.
\end{prop}
\begin{proof}
  cf. Case C above.
\end{proof}
 
\end{document}